\documentclass{my-aims}

\usepackage{amsmath}
\usepackage{paralist}
\usepackage{graphics}
\usepackage{graphicx}
\usepackage{subfig}
\usepackage[colorlinks=true, urlcolor=blue, citecolor=red]{hyperref}

\def\ppages{X--XX}

\newtheorem{theorem}{Theorem}[section]

\theoremstyle{definition}

\def\R{\mathbb{R}}

\title[Optimal control of a delayed HIV model]{Optimal control of a delayed HIV model}

\author[F. Rodrigues, C. J. Silva, D. F. M. Torres and H. Maurer]{}

\subjclass{Primary: 34C60, 49K15; Secondary: 92D30.}

\keywords{HIV, intracellular and pharmacological time delays, stability,
qualitative investigation and simulation of models, optimal control.}

\email{fmgrodrigues@ua.pt}
\email{cjoaosilva@ua.pt}
\email{delfim@ua.pt}
\email{maurer@math.uni-muenster.de}

\thanks{$^*$Corresponding author: delfim@ua.pt}

\begin{document}

\maketitle


\centerline{\scshape Filipe Rodrigues, Cristiana J. Silva, Delfim F. M. Torres$^*$}
\medskip
{\footnotesize
\centerline{Center for Research and Development in Mathematics and Applications (CIDMA)}
\centerline{Department of Mathematics, University of Aveiro, 3810-193 Aveiro, Portugal}}

\medskip

\centerline{\scshape Helmut Maurer}
\medskip
{\footnotesize
\centerline{Institute of Computational and Applied Mathematics}
\centerline{University of M\"{u}nster, D-48149 M\"{u}nster, Germany}
}

\bigskip


\begin{abstract}
We propose a model for the human immunodeficiency virus type~1 (HIV-1)
infection with intracellular delay and prove the local asymptotical
stability of the equilibrium points. Then we introduce a control function
representing the efficiency of reverse transcriptase inhibitors
and consider the pharmacological delay associated to the control.
Finally, we propose and analyze an optimal control problem
with state and control delays. Through numerical simulations,
extremal solutions are proposed for minimization
of the virus concentration and treatment costs.
\end{abstract}


\section{Introduction}
\label{IntroSection}

Infection by human immunodeficiency virus type~1 (HIV-1)
has many quantitative features \cite{PerelsonNelson1999}.
Mathematical models for HIV infection can provide insights
into the dynamics of viral load in vivo and may play a significant
role in the development of a better understanding of HIV/AIDS
and drug therapies \cite{ZhuZou:DelayHIV:DCDS:2009}.
Cytotoxic T lymphocytes (CTLs) play a critical role in antiviral
defense by attacking virus-infected cells. It is believed that
CTLs are the main host immune factor that determine virus load
\cite{Nowak:Bangham:HIV:Science:1996}. When HIV invades the body,
it targets the CD4+ T cells. These cells can be considered the
command centers of the immune system. The CTLs are cells that
set out to eliminate infection by killing infected cells \cite{CulshawRuanSpiteri2004}.
Several mathematical models have been proposed for HIV-1 infection with CTLs response:
see, e.g., \cite{ArnaoutNowakWodarzHIV2000,CulshawRuanSpiteri2004,%
Nowak:Bangham:HIV:Science:1996, WangWangLiuCMA2006} and references cited therein.

Time delay plays an important role in the dynamics of HIV infection.
Intracellular delay, that is, the delay between initial infection
of a cell by HIV and the release of new virions, was considered
in the models proposed by \cite{CulshawCD4Tdelay,Herz:HIV:delay:1996,%
Mittler:delay:AIDS:1999,Mittler:HIVdelay:Math:Bio:1998,NelsonMurray:delay:HIV:2000,%
NelsonPerelson:delay:HIV:2002,Tam:delay:1999,ZhuZou:DelayHIV:DCDS:2009}.
Here, we enrich the undelayed mathematical model proposed by \cite{Nowak:Bangham:HIV:Science:1996},
which considers the action of CTLs in the immune system, by introducing a discrete
time delay that represents an intracellular delay. State delays for such type of models 
have been already introduced, e.g., in \cite{Haffat:Yousfi:OCHIVdelay:2012}. 
However, in our case we also model the important pharmacological delay that occurs between
the administration of drug and its appearance within cells, due to the time required 
for drug absorption, distribution, and penetration into the target cells
\cite{PerelsonEtAllHIVdynamicsScience1996}. In the context of anticancer therapy, 
the idea to represent delay effects in drug kinetics and dynamics was presented 
in \cite{[3]} and developed in \cite{[4]}.

Optimal control is a branch of mathematics developed
to find optimal ways to control a dynamic system
\cite{Cesari_1983,Fleming_Rishel_1975,Pontryagin_et_all_1962}.
Optimal control theory has been applied with success to HIV models: see, e.g.,
\cite{CulshawRuanSpiteri2004,Haffat:Yousfi:OCHIVdelay:2012,%
Kirschner:Lenhart:OC:HIV:JMB:1996,MyID:355,MyID:318} 
and references cited therein. Here, we introduce a control function,
which represents the efficiency of reverse transcriptase inhibitors,
and consider a delay in the control function representing the pharmacological delay. 
Our aim is to determine the control function that minimizes the concentration
of virus and the treatment costs. To the best of our knowledge, this is the first time
an optimal control HIV problem with delay in state and control variables is investigated.

The paper is organized as follows. The model with intracellular delay is formulated 
in Section~\ref{Section:delay:model} and local stability is proved for any time delay.
In Section~\ref{sec:delay:OC:problem}, we introduce a control function in the
delayed model of Section~\ref{Section:delay:model} and analyze an optimal control
problem with intracellular and pharmacological delays. Section~\ref{sec:numerical}
is devoted to numerical simulations for the stability of the equilibrium points
and the computation of extremals for the optimal control problem with state and control delays.
We compare the extremal of our optimal control problem with state and control delays 
with the solutions of the uncontrolled problem and the control problem
with delay in the state variable only. We end with Section~\ref{sec:discussion},
where we discuss the established results.


\section{Intracellular delayed mathematical model}
\label{Section:delay:model}

In this section, we propose a delayed mathematical model for HIV-1 infection.
We consider the undelayed model proposed by \cite{Nowak:Bangham:HIV:Science:1996}
and introduce a discrete intracellular time delay. The model considers four state variables:
$Z(t)$ represents the concentration of uninfected cells, $I(t)$ represents
the concentration of infected cells, $V(t)$ represents the concentration
of free virus particles, and $T(t)$ represents the concentration of CTLs at time $t$.
The following assumptions are made to describe the cell dynamics
\cite{Nowak:Bangham:HIV:Science:1996}: uninfected cells are produced at a constant
rate $\lambda$, and die at a rate $m Z$. Infected cells are produced
from uninfected cells and free viruses at a rate $rVZ$ and die at rate $u I$
(the average lifetime of an infected cell is $1/u$). Free viruses are produced
from infected cells at rate $kI$ and declines at rate $vV$
(the average lifetime of a free virus particle is $1/v$). The rate
of CTLs proliferation in response to antigen is given by $a I T$.
In the absence of stimulation, CTLs decay at rate $n T$. Infected
cells are killed by CTLs at rate $s I T$.
The intracellular delay, $\tau$, represents the time needed for
infected cells to produce virions after viral entry 
\cite{Haffat:Yousfi:OCHIVdelay:2012,ZhuZou:DelayHIV:DCDS:2009},
called the eclipse phase \cite{Pawelek:HIV:2delays:MBio:2012}.
The model we propose is given by the following
system of ordinary differential equations:
\begin{equation}
\label{sistemaT}
\left\{
\begin{array}{lcr}
\dot{Z}(t)=\lambda-mZ(t)-rV(t)Z(t), \\
\dot{I}(t)=rV(t-\tau)Z(t-\tau)-uI(t)-sI(t)T(t), \\
\dot{V}(t)=kI(t)-vV(t), \\
\dot{T}(t)=aI(t)T(t)-nT(t).
\end{array}\right.
\end{equation}
The initial conditions for system \eqref{sistemaT} are
\begin{equation}
\label{eq:init:cond:moddelay}
Z(\theta) = \varphi_1(\theta), \quad I(\theta) = \varphi_2(\theta),
\quad V(\theta) = \varphi_3(\theta), \quad T(\theta) = \varphi_4(\theta),
\end{equation}
$-\tau \leq \theta \leq 0$, where $\varphi
=\left(\varphi_1, \varphi_2, \varphi_3, \varphi_4\right)^T \in C$
with $C$ the Banach space $C \left([-\tau, 0], {\mathbb{R}}^4 \right)$
of continuous functions mapping the interval $[-\tau, 0]$ into ${\mathbb{R}}^4$.
The usual local existence, uniqueness and continuation results apply
\cite{Hale_Lunel_book1993,YKuang_1993}. Moreover, from biological meaning,
we further assume that the initial functions are nonnegative:
\begin{equation}
\label{eq:nonneg:IC}
\varphi_i(\theta) \geq 0, \quad \text{for}  \quad
\theta \in [-\tau, 0], \quad i = 1, \ldots, 4.
\end{equation}
From \cite[Theorem 2.1]{ZhuZou:DelayHIV:DCDS:2009}, it follows that
all solutions of \eqref{sistemaT} satisfying \eqref{eq:init:cond:moddelay}
and \eqref{eq:nonneg:IC} are bounded for all time $t \geq 0$,
which ensures not only local existence
but the existence of a unique solution
$\left( Z(t), I(t), V(t), T(t)\right)$ of \eqref{sistemaT} with initial
conditions \eqref{eq:init:cond:moddelay}--\eqref{eq:nonneg:IC}
for all time $t \geq 0$.

The equilibrium points are independent of the delays.
Their stability depends, however, on the delays.
The equilibrium points of \eqref{sistemaT} are studied in
\cite{GlobalStabVirus:Pruss:2008,ZhuZou:DelayHIV:DCDS:2009}.
System \eqref{sistemaT} has an infection-free equilibrium
$E_0=\left(\frac{\lambda}{m},0,0,0\right)$, which is the
only biologically meaningful equilibrium,
if $R_0=\frac{k\lambda r}{muv} < 1$.
Let $R_1=\frac{knr}{mav}$. If $1<R_0<1+R_1$, then system
\eqref{sistemaT} has a unique CTL-inactivated infection
equilibrium $E_1$ given by
$$
E_1= \left(\frac{uv}{kr},\frac{k\lambda r-muv}{kru},
\frac{k\lambda r-muv}{vru},0\right).
$$
Whenever $R_0>1+R_1$, system \eqref{sistemaT} has also
a CTL-activated infection equilibrium $E_2$ given by
$$
E_2= \left(\frac{a\lambda v}{amv+knr},\frac{n}{a},\frac{kn}{av},
\frac{ak\lambda r-amuv-knru}{amvs+knrs}\right).
$$
The proofs of these facts are found in
\cite{GlobalStabVirus:Pruss:2008,ZhuZou:DelayHIV:DCDS:2009}.
Here we prove the local asymptotic stability of the equilibrium points
$E_0$, $E_1$ and $E_2$ for any time delay $\tau$.

\begin{theorem}[Local stability of the equilibrium points of \eqref{sistemaT}]
\label{theo:localstab}
If $R_0 > 1$, then the infection-free equilibrium $E_0$
is unstable for any time-delay $\tau \geq 0$.
If $R_0 < 1$, then $E_0$ is locally asymptotically stable for any time-delay $\tau \geq 0$.
If $R_0=1$, then we have a critical case.
If $R_0>1+R_1$, then the CTL-inactivated infection
equilibrium $E_1$ is unstable for any time-delay $\tau \geq 0$.
If $R_0<1+R_1$, then $E_1$ is locally asymptotically stable for any time-delay $\tau \geq 0$.
If $R_0>1+R_1$, then the CTL-activated infection equilibrium $E_2$
is locally asymptotically stable for any time-delay $\tau \geq 0$.
\end{theorem}

\begin{proof}
Consider the following coordinate transformation:
$$
x(t)=Z(t)-\bar{Z}, \qquad y(t)=I(t)-\bar{I},
\qquad w(t)=V(t)-\bar{V}, \qquad q(t)=T(t)-\bar{T},
$$
where $(\bar{Z},\bar{I}, \bar{V},\bar{T})$ denotes
any equilibrium point of system \eqref{sistemaT}.
The linearized system of \eqref{sistemaT} is of form
\begin{equation}
\label{sistemaL}
\left\{
\begin{array}{lcr}
x(t)=-(m+r\bar{V})x(t)-r\bar{Z}w(t),\\
y(t)=r\bar{V}x(t-\tau) - (u+s\bar{T})y(t) +r\bar{Z}w(t-\tau) -s\bar{I}p(t),\\
w(t)=ky(t)-vw(t),\\
p(t)=a\bar{T}y(t)+(a\bar{I}-n)p(t).
\end{array}\right.
\end{equation}
We can express system \eqref{sistemaL} in matrix form as follows:
$$
\frac{d}{dt}
\begin{pmatrix}
x(t)\\ y(t) \\ w(t) \\p(t)
\end{pmatrix}
= A_1
\begin{pmatrix}
x(t)\\ y(t) \\ w(t) \\p(t)
\end{pmatrix}
+  A_2
\begin{pmatrix}
x(t-\tau)\\ y(t-\tau) \\ w(t-\tau) \\p(t-\tau)
\end{pmatrix},
$$
where $A_1$ and $A_2$ are $4\times 4$ matrices given by
$$
A_1=
\begin{pmatrix}
-m-r\bar{V} & 0 & -r\bar{Z} & 0\\
0 & -u-s\bar{T} & 0 & -s\bar{I}\\
0 & k & -v & 0\\
0 & a\bar{T} & 0 & a\bar{I}-n			
\end{pmatrix},
\quad
A_2=
\begin{pmatrix}
0 & 0 & 0 & 0\\
r\bar{V} & 0 & r\bar{Z} & 0\\
0 & 0 & 0 & 0\\
0 & 0 & 0 & 0
\end{pmatrix}.
$$
The characteristic equation of system \eqref{sistemaL}
for any equilibrium point is given by
\begin{equation}
\label{Eq}
\Delta(y)=\left|yI_d-A_1-A_2e^{-\tau y}\right| = 0
\end{equation}
(see, e.g., \cite{YKuang_1993}), where $I_d$ denotes 
the identity matrix of dimension 4, that is,
\begin{equation*}
\left| \begin{matrix}
m+r\bar{V}+y & 0 & r\bar{Z} & 0\\
-\bar{V}re^{-\tau y} & u+\bar{T}s+y & -\bar{Z}re^{-\tau y} & \bar{I}s\\
0 & -k & v+y & 0\\
0 & -\bar{T}a & 0 & y+n-a\bar{I}			
\end{matrix} \right| = 0.
\end{equation*}
(i) \emph{Stability of the infection-free equilibrium $E_0$}.
The characteristic equation at $E_0$ is given by
\begin{equation}
\label{e0}
(y+n)(y+m)\left( (u+y)(v+y) -\frac{k\lambda r}{m}e^{-\tau y}\right)=0.
\end{equation}
Assume that $\tau =0$. In this case, the equation \eqref{e0} becomes
\begin{equation}
\label{e0t}
(y+n)(y+m)\left( (u+y)(v+y) -\frac{k\lambda r}{m}\right)=0.
\end{equation}
We need to prove that all the roots of the characteristic equation
have negative real parts. It is easy to see that $y_1=-n$ and $y_2=-m$
are roots of equation \eqref{e0t} and both are real negative roots.
Thus, we just need to consider the third term of the above equation. Let
\begin{equation*}
p(y):=y^2+(u+v)y + uv-\frac{k\lambda r}{m}=0.
\end{equation*}
Using the Routh--Hurwitz criterion, we know that all roots
of $p(y)$ have negative real parts if and only if the coefficients $a_i$
of $p(y)$ are strictly positive. In our case,
\begin{gather*}
a_1=1>0,\\
a_2=u+v>0,\\
a_3=uv-\frac{k\lambda r}{m} >0 \text{ if and only if } 
R_0 = \frac{k\lambda r}{muv} < 1.  
\end{gather*}
Hence, if $R_0 < 1$, then all roots of the characteristic equation
\eqref{e0t} have negative real parts. Therefore, $E_0$
is locally asymptotically stable for $\tau =0$.
Suppose now that $\tau > 0$. To prove the stability of $E_0$ we use Rouch\'e theorem, 
so we need to prove that all the roots of the characteristic equation \eqref{e0}
cannot intersect the imaginary axis, i.e., the characteristic equation cannot
have pure imaginary roots. Suppose the reverse, i.e., that there exists
$w \in \mathbb{R}$ such that $y=wi$ is a solution of \eqref{e0}. Replacing
$y$ in the third term of \eqref{e0}, we get
\begin{equation*}
(wi)^2+(u+v)wi + uv-\frac{k\lambda r}{m}e^{-\tau wi}=0.
\end{equation*}
Note that we do not need to consider the full equation \eqref{e0}
because we already know that the remaining part of this equation
has just two real negative solutions. By using the Euler formula
and separating the real and imaginary parts of the above equation,
we obtain that
$$
\left\lbrace
\begin{array}{l}
-w^2+ uv = \frac{k\lambda r}{m} \cos(w\tau), \\[0.3cm]
(u+v)w=-\frac{k\lambda r}{m} \sin(w\tau).
\end{array} \right.
$$
By adding up the squares of both equations and using the
fundamental trigonometric formula, we obtain that
$$
(-w^2+ uv)^2 + (u+v)^2w^2 - \left(\frac{k\lambda r}{m}\right)^2=0,
$$
which is the same as
$$
w^4 + (u^2+v^2)w^2 + u^2v^2  - \left(\frac{k\lambda r}{m}\right)^2=0
$$
and equivalent to
$$
w^2=\frac{1}{2}\left(-(u^2+v^2) + \sqrt{\left(u^2+v^2\right)^2
-4\left(u^2v^2-\left(\frac{k\lambda r}{m}\right)^2\right)}\right).
$$
If $R_0 < 1$, then $muv-k\lambda r>0$, which implies $(muv)^2-(k\lambda r)^2>0$.
Consequently,
$$
u^2v^2-\left(\frac{k\lambda r}{m}\right)^2>0
$$
and
$$
\sqrt{\left(u^2+v^2\right)^2
-4\left(u^2v^2-\left(\frac{k\lambda r}{m}\right)^2\right)} < u^2+v^2.
$$
Hence, we have $w^2<0$, which is a contradiction. Therefore, we proved that
if $R_0 < 1$, then the characteristic equation \eqref{e0} cannot have pure
imaginary roots and the infection-free equilibrium $E_0$
is locally asymptotically stable for any strictly positive time-delay.
Suppose now that $R_0 > 1$. We know that the characteristic equation \eqref{e0}
has two real negative roots $y=-n$ and $y=-m$. Thus, we need to check
if the remaining roots of
\begin{equation*}
q(y):=(u+y)(v+y) -\frac{k\lambda r}{m}e^{-\tau y}
\end{equation*}
have negative real parts. It is easy to see that $q(0)=uv-\frac{k\lambda r}{m}<0$,
because we are assuming $R_0 > 1$. On the other hand,
$\lim\limits_{y\rightarrow+\infty} q(y)=+\infty$. Therefore, by continuity
of $q(y)$, there is at least one positive root of the characteristic equation
\eqref{e0}. Hence, we conclude that $E_0$ is unstable.
Finally, we need to analyse the case $R_0=1$, i.e., $muv=k\lambda r$.
In this case the characteristic equation \eqref{e0} becomes
\begin{equation}
\label{0}
(y+n)(y+m)\left(y^2 + (u+v)y +uv-uv e^{-\tau y}\right)=0.
\end{equation}
To prove the stability, we need to check again if all the roots
of the above equation have negative real parts. Note that $y=0$,
$y=-n$ and $y=-m$ are solutions of this equation, so we just need
to prove that the remaining roots cannot have nonnegative real parts.
Assuming that $y=a+bi$ with $a\geq 0$ is a solution of the above equation, then
\begin{equation*}
(a+bi)^2+(u+v)(a+bi) +uv-uv e^{-\tau (a+bi)}=0.
\end{equation*}
By using the Euler formula and by separating the real and imaginary parts, we get
$$
\left\lbrace
\begin{array}{l}
a^2-b^2+ (u+v)a+uv = uve^{-\tau a}\cos(\tau b),\\
2ab+(u+v)b=-uve^{-\tau a} \sin(w\tau).
\end{array} \right.
$$
Adding up the squares of both equations and using
the fundamental trigonometric formula, we obtain
$$
\left(a^2-b^2+ (u+v)a+uv\right)^2 + \left(2ab+(u+v)b\right)^2
= \left(uve^{-\tau a}\right)^2\leq u^2v^2,
$$
which is a contradiction because
\begin{equation*}
\begin{split}
(a^2-b^2&+ (u+v)a+uv)^2 + \left(2ab+(u+v)b\right)^2\\
&= {a}^{4}+2\,{a}^{3}u+2\,{a}^{3}v+2\,{a}^{2}{b}^{2}+{a}^{2}{u}^{2}
+4\,{a}^{2}uv+{a}^{2}{v}^{2}+2\,a{b}^{2}u\\
&\quad +2\,a{b}^{2}v+2\,a{u}^{2}v
+2\,au{v}^{2}+{b}^{4}+{b}^{2}{u}^{2}+{b}^{2}{v}^{2}+{u}^{2}{v}^{2}
> u^2v^2.
\end{split}
\end{equation*}
This proves that $0$ is the unique root
of \eqref{0} that does not have negative real part.\\
(ii) \emph{Stability of CTL-inactivated infection equilibrium $E_1$}.
Assume that $R_0>1$. The characteristic equation \eqref{Eq} at 
$E_1=\left(Z_1,I_1,V_1,T_1\right)=\left(\frac{uv}{kr},\bar{I},\bar{V},0\right)$
is given by
\begin{equation}
\label{aa}
\left(n-\bar{I}a+y\right)\left(y^3+Ay^2+By+C-(Dy+E)e^{-\tau y}\right)=0,
\end{equation}
where
$A=m+u+v+\bar{V}r$, $B=\bar{V}ru+\bar{V}rv+mu+mv+uv$, $C=muv + \bar{V}ruv$,
$D=uv$, and $E= muv$. Note that $y=\bar{I}a-n$ is a solution of \eqref{aa}.
If $1<R_0<1+R_1$, then
$$
\frac{k\lambda r}{muv}<1+\frac{knr}{amv}.
$$
After some basic simplifications, we have
$\bar{I}a-n <0$. Hence, if $R_0>1+R_1$, then the characteristic equation
\eqref{aa} has a positive root and, consequently, the equilibrium
$E_1$ is not locally asymptotically stable. On the other hand,
if $1<R_0<1+R_1$, then $y=\bar{I}a-n$ is a real negative root
of the characteristic equation \eqref{aa} and we just need to analyze the equation
\begin{equation}
\label{qa}
y^3+Ay^2+By+C-(Dy+E)e^{-\tau y}=0.
\end{equation}
Consider $\tau =0$. From equation \eqref{qa} we have
\begin{equation}
\label{a}
y^3+Ay^2+(B-D)y+(C-E)=0,
\end{equation}
where
$A=m+u+v+\bar{V}r >0$, $B-D=\bar{V}ru+\bar{V}rv+mu+mv >0$, $C-E=\bar{V}ruv>0$
and $A(B-D)>(C-E)$. Therefore, from the Routh--Hurwitz criterion, it follows
that all roots of \eqref{a} have negative real parts. Hence, $E_1$
is locally asymptotically stable for $\tau=0$. Let $\tau > 0$.
Suppose that \eqref{qa} has pure imaginary roots, $wi$.
By replacing $y$ in \eqref{qa} by $wi$, we get
\begin{equation*}
-Aw^2+C-E\cos(w\tau)-Dw\sin(w\tau)+i(-w^3+Bw-Dw\cos(w\tau)+E\sin(w\tau))=0.
\end{equation*}
If we separate the real and imaginary parts, then we obtain
$$
\left\lbrace
\begin{array}{l}
-Aw^2+C= E\cos(w\tau)+Dw\sin(w\tau),\\
-w^3+Bw= Dw\cos(w\tau)-E\sin(w\tau).
\end{array} \right.
$$
By adding up the squares of both equations, and using
the fundamental trigonometric formula, we obtain that
$$
E^2+D^2w^2=(-Aw^2+C)^2+(-w^3+Bw)^2,
$$
which is equivalent to
$$
w^6+(A^2-2B)w^4+(B^2-2AC-D^2)w^2+(C^2-E^2)=0.
$$
Since
\begin{gather*}
A^2-2B=m^2+u^2+v^2+(\bar{V}r)^2+2\bar{V}mr>0,\\
B^2-2AC-D^2=
(\bar{V}ru)^2+\bar{V}rv)^2 +(mu)^2+(mv)^2+2\bar{V}rm(u^2+v^2)>0,\\
C^2-E^2=(\bar{V}ruv)^2+2\bar{V}ru^2v^2m>0,
\end{gather*}
we have that the left hand-side of equation \eqref{a} is strictly positive,
which implies that this equation is not possible. Therefore, \eqref{aa}
does not have imaginary roots, which implies that $E_1$ is locally asymptotically
stable for any time delay $\tau \geq 0$.\\
(iii) \emph{Stability of CTL-activated infection equilibrium $E_2$}.
Assume $R_0>1+R_1$. The characteristic equation \eqref{Eq}
at $E_2=\left(Z_2,I_2,V_2, T_2\right)
=\left(\bar{Z},\frac{n}{a},\bar{V},\bar{T}\right)$
becomes
\begin{equation}
\label{ewe}
(y+m+\bar{V}r)(v+y)\left( \bar{T}ns+y(u+y+\bar{T}s)\right)
=(y+m)v\frac{\bar{Z}kry}{v}e^{-\tau y}.
\end{equation}
Suppose that there is a $wi$, $w \in \R$,
such that $y=wi$ is root of equation \eqref{ewe}. Then,
\begin{equation*}
(wi+m+\bar{V}r)(v+wi)\left( \bar{T}ns+wi(u+wi+\bar{T}s)\right)
=(wi+m)v\bar{Z}\frac{kr}{v}wie^{-\tau wi},
\end{equation*}
which implies
\begin{equation}
\label{ee}
\left|wi+m+\bar{V}r\right|^2  \left|v+wi\right|^2
\left| \bar{T}ns+wi(u+wi+\bar{T}s)\right|^2
=\left|wi+m\right|^2 \left|v\right|^2\left|\frac{\bar{Z} k r}{v}wi\right|^2.
\end{equation}
Since
$\left|wi+m+\bar{V}r\right|^2 >\left|wi+m\right|^2$ and
$\left|v+wi\right|^2 >\left|v\right|^2$, 
it follows from \eqref{ee} that
$$
\left| \bar{T}ns+wi(u+wi+\bar{T}s)\right|^2
\geq\left|\bar{Z}\frac{kr}{v}wi\right|^2.
$$
We conclude that the left hand-side of \eqref{ee}
is always strictly greater than the right hand-side,
which implies that this equation is impossible.
Hence, the solutions of the characteristic equation \eqref{ewe}
cannot be pure imaginary. Therefore, by Rouch\`e theorem,
$E_2$ is locally asymptotically stable
for any time-delay $\tau \geq 0$.
\end{proof}


\section{Optimal control of the HIV model
with intracellular and pharmacological delays}
\label{sec:delay:OC:problem}

In the human system, RNA molecules are produced from DNA.
Nevertheless, there are enzymes that make the reverse process, i.e.,
they can obtain DNA molecules from RNA. Such an enzyme is called
a \emph{reverse transcriptase}. One kind of such enzymes are found in HIV-1.
As a result, when a virus particle infects a T-cell,
it comes into the kernel of the cell and makes the reverse
transcriptase process converting the RNA viral molecules into DNA
viral molecules, which are then combined with DNA molecules of the CTLs.
Hence, CTLs work to create new viruses instead of doing the defense
job they are supposed to do in the immune system. Nowadays, there
are drugs that can inhibit the reverse transcriptase,
which allow the CTLs to keep their natural work.
In this section, we formulate an optimal control problem for HIV-1
infection, with time delay in state and control variables,
and derive extremals for the minimization of virus
by the use of drugs that inhibit the reverse transcriptase of CTLs.

We introduce a control function $c(t)$ in model \eqref{sistemaT},
$t \in [0, t_f]$, that represents the efficiency
of the reverse transcriptase inhibitors, which block a new infection.
Due to the importance of the pharmacological delay in the HIV treatment,
we consider a discrete time delay in the control variable $c(t)$,
denoted by $\xi$, which represents the delay that occurs
between the administration of a drug and its appearance within the cells,
due to the time required for drug absorption, distribution,
and penetration into the target cells \cite{PerelsonEtAllHIVdynamicsScience1996}.
We propose the following control system with discrete time delay
in the state and control variables:
\begin{equation}
\label{eq:model:delay:control}
\left\{
\begin{array}{lcr}
\dot{Z}(t)=\lambda-mZ(t)-(1-c(t-\xi))rV(t)Z(t),\\
\dot{I}(t)=(1-c(t - \xi))rV(t-\tau)Z(t-\tau)-uI(t)-sI(t)T(t),\\
\dot{V}(t)=kI(t)-vV(t),\\
\dot{T}(t)=aI(t)T(t)-nT(t).
\end{array}\right.
\end{equation}
The initial conditions for the state variables $I$ and $T$ and,
due to the delays, initial functions for the state variables
$Z$ and $V$ and control $c$, are given by
\begin{equation}
\label{eq:initcond:delays}
\begin{split}
I(0) &= 3, \quad T(0) = 20,\\
Z(t) &\equiv 45, \quad V(t) \equiv 75,
\quad -\tau \leq t \leq 0,\\
c(t) &\equiv 0, \quad -\xi \leq t < 0.
\end{split}
\end{equation}
We note that values \eqref{eq:initcond:delays} are the only ones
that are considered in our numerical simulations (Section~\ref{sec:numerical}).
The control function $c(t)$ is bounded between 0 and 1.
If it takes the value 0, then the drug therapy for the transcriptase
reversion has no efficacy. If the control takes the value 1, then
it will be 100\% effective. Precisely, we consider following set
of admissible control functions:
\begin{equation}
\label{eq:adm:controls}
\Theta = \biggl\{ c(\cdot) \in L^1\left([0, t_f], \R \right) \,
| \,  0 \leq c (t) \leq 1 \, ,  \, \forall \, t \in [0, t_f] \, \biggr\} .
\end{equation}
We consider the $L^1$ objective functional
\begin{equation}
\label{costfunction}
J(c(\cdot)) = \int_0^{t_f} \left[ V(t)
+ w\cdot c(t) \right] dt  \quad (\mbox{weight parameter} \;\, w \geq 1),
\end{equation}
which measures the concentration of virus and the treatment costs
for the period of time under study. The optimal control problem
consists in determining a control function
$c(\cdot) \in L^1\left([0, t_f], \R \right)$ that minimizes the cost functional
\eqref{costfunction} subject to the control system \eqref{eq:model:delay:control},
initial conditions \eqref{eq:initcond:delays}
and control constraints \eqref{eq:adm:controls}.
In Section~\ref{sec:4.2}, we present numerical solutions for three cases 
of delays $\tau$ and $\xi$ and weights $w=1$ and $w=5$.
To apply the optimality conditions given by the Minimum Principle
for Multiple Delayed Optimal Control Problems
of \cite[Theorem~3.1]{Goellmann-Maurer-14}, we introduce the delayed
state variables $\zeta(t) =Z(t - \tau)$, $\eta(t) = V(t - \tau)$
and the control variable $\omega(t) = c(t - \xi)$. Using the adjoint
variable $\lambda = \left( \lambda_Z, \lambda_I, \lambda_V,
\lambda_T \right) \in \R^4$, the Hamiltonian for the cost functional
\eqref{costfunction} and the control system
\eqref{eq:model:delay:control} is given by
\begin{multline*}
H(Z, \zeta, I, V, \eta, T, \lambda, c, \omega)
= V + w c + \lambda_Z \left( \lambda-mZ -(1-\omega)r V Z \right)\\
+ \lambda_I \left((1-\omega)r \eta \zeta - u I 
- sI T \right) + \lambda_V \left( kI -v V \right)
+ \lambda_T \left(aI T -nT\right).
\end{multline*}
The adjoint equations are given by
\begin{equation*}
\begin{cases}
\dot{\lambda}_Z(t) = - H_Z[t] - \chi_{[0, t_f-\tau]} H_\zeta[t + \tau],\\
\dot{\lambda}_V(t) = - H_V[t]  - \chi_{[0, t_f-\tau]} H_\eta[t + \tau],\\
\dot{\lambda}_I(t) = - H_I[t], \quad \dot{\lambda}_T(t) = - H_T[t] ,
\end{cases}
\end{equation*}
where the subscripts denote partial derivatives and
$\chi_{[0, t_f-\tau]}$ is the characteristic function
in the interval $[0, t_f - \tau]$ (see \cite{Goellmann-Maurer-14}).
Since the terminal state is free, i.e.,
$(Z(t_f), I(t_f), V(t_f), T(t_f)) \in \R^4$,
the transversality conditions are
$$
\lambda_Z(t_f)=\lambda_I(t_f) = \lambda_V(t_f) = \lambda_T(t_f) = 0.
$$
To characterize the optimal control $c$, we introduce
the following \emph{switching function}:
\begin{equation}
\label{eq:sf}
\begin{split}
\phi(t) &= H_c[t] + \chi_{[0, t_f - \xi]} H_\omega[t + \xi]\\
&= \begin{cases}
1 + \lambda_Z(t +\xi) r V(t + \xi) Z(t + \xi)
- \lambda_I(t +\xi) r \eta(t + \xi) \zeta(t + \xi)
\\
\hspace{9mm} \text{for} \quad 0 \leq t \leq t_f - \xi,\\
1 \qquad \text{for} \quad t_f - \xi \leq t \leq t_f.
\end{cases}
\end{split}
\end{equation}
The minimality condition of the Minimum Principle
\cite[Theorem~3.1]{Goellmann-Maurer-14} gives the control law
\begin{equation}
\label{control-law}
c(t) = \left\{
\begin{array}{rcl}
1 &&\mbox{if} \quad \phi(t) < 0,  \\[1mm]
0 && \mbox{if} \quad \phi(t) > 0,  \\[1mm]
{\rm singular} &&  \mbox{if} \quad \phi(t) = 0
\;\; \mbox{on} \; I_s \subset [0,t_f].
\end{array}
\right.
\end{equation}

Similar arguments can also be used to solve related optimal control problems, 
e.g., one may consider an additional constraint on the final
virus concentration or inclusion of the final values
of this concentration in the cost functional.


\section{Numerical simulations}
\label{sec:numerical}

In this section, we study numerically the stability
of the delayed model \eqref{sistemaT} proposed
in Section~\ref{Section:delay:model} and the solution
of the optimal control problem proposed
in Section~\ref{sec:delay:OC:problem}.
We consider the initial conditions \eqref{eq:initcond:delays}
and the parameter values as given in Table~\ref{table:parameter},
which are based on \cite{Haffat:Yousfi:OCHIVdelay:2012}.
\begin{table}[htbp]
\begin{center}
\begin{tabular}{c|c|c} \hline
Parameter & Value & Units\\[0.1cm] \hline
$\lambda$ & $5 $ &$ \; day^{-1}mm^{-3}$ \\
$m$& $0.03 $ &$\; day^{-1}$\\
$r$& $0.0014 $ &$\; mm^3 virion^{-1} day^{-1}$\\
$u$&  $0.32 $ &$\; day^{-1}$\\
$s$&  $0.05 $ &$\; mm^3 day^{-1}$\\
$k$& $153.6 $ &$\; day^{-1}$\\
$v$& $1 $ &$\; day^{-1}$\\
$a$& $0.2  $ &$\; mm^3 day^{-1}$\\
$n$&  $0.3 $ &$\; day^{-1}$\\
$t_f$&  $50 $ &$\; day$\\
$\tau$&  $0.5 $ &$\; day$\\	
$\xi$&  $0.2 $ &$\; day$\\ \hline
\end{tabular}
\caption{Parameter values.}
\label{table:parameter}
\end{center}
\end{table}


\subsection{Stability of the delayed HIV model}

Considering the parameter values from Table~\ref{table:parameter},
we have the following values for the thresholds $R_0$ and $R_1$
of Section~\ref{Section:delay:model}:
\begin{equation*}
R_0 = 112 \quad \quad \text{and} \quad \quad R_1 = 10.752.
\end{equation*}
From Theorem~\ref{theo:localstab}, the CTL-activated
infection equilibrium
$$
E_2 = \left( 14.182, 1.5, 230.4, 54.5939 \right)
$$
of system \eqref{sistemaT} is locally asymptotically stable
for any time delay $\tau \geq 0$. In Figure~\ref{fig:endemic:equi},
we observe the stability of system \eqref{sistemaT} in a time interval
of $500$ days and a time delay of $0.5$ days ($\tau = 0.5$).
\begin{figure}[htb]
\centering
\subfloat[\footnotesize{State variables}]{\label{statevar:endemic}
\includegraphics[width=0.45\textwidth]{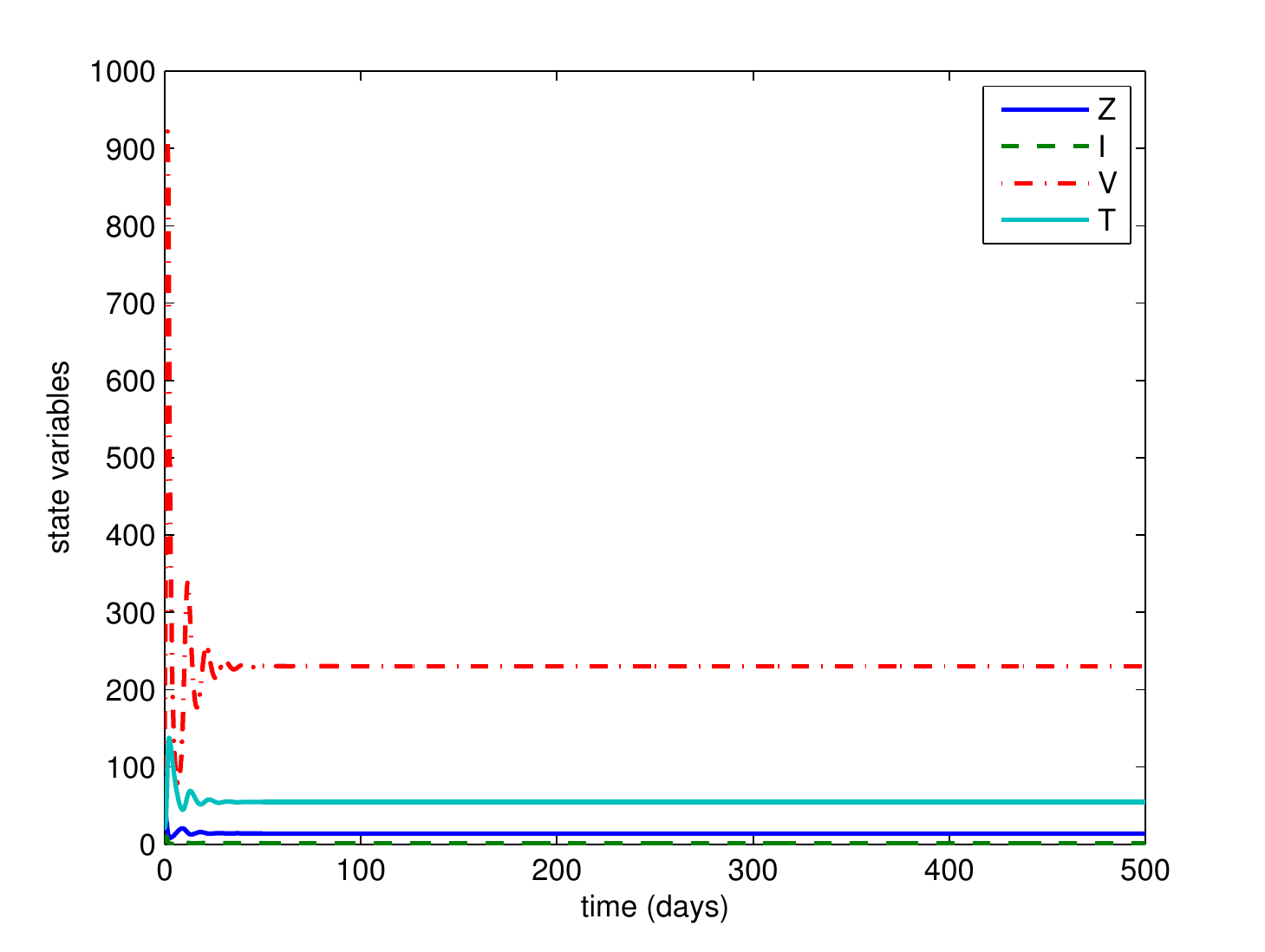}}
\subfloat[\footnotesize{$(T, V)$}]{\label{statevar:VT:endemic}
\includegraphics[width=0.45\textwidth]{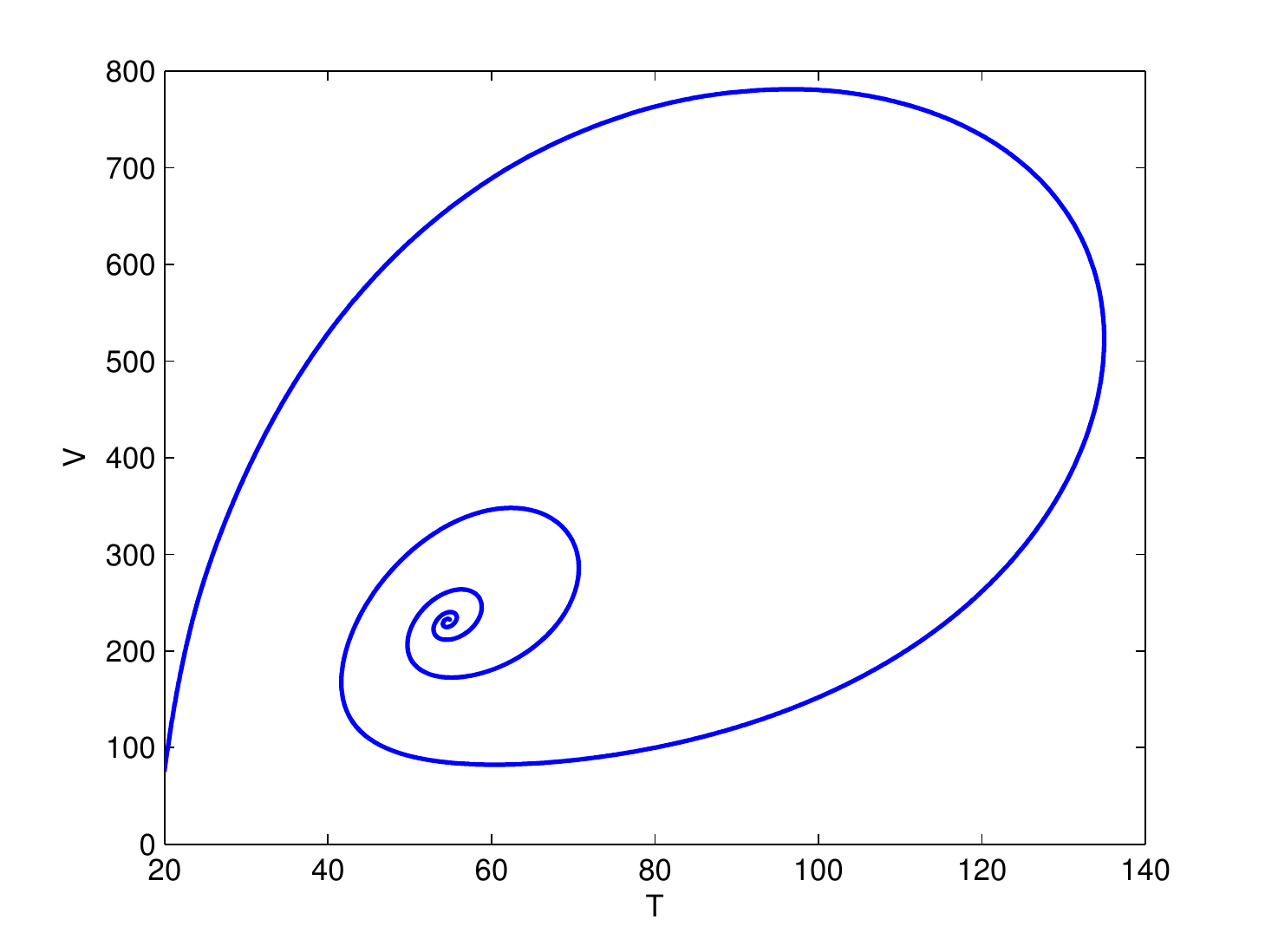}}
\caption{Endemic equilibrium $E_2$ for the parameter values
of Table~\ref{table:parameter} and time delay $\tau=0.5$.}
\label{fig:endemic:equi}
\end{figure}
In Figure~\ref{fig:statevars:Delay:yes:no}, we compare the behavior
of system \eqref{sistemaT} for $\tau = 0$ (no delay)
and delay $\tau = 0.5$.
\begin{figure}[htb]
\centering
\subfloat[\footnotesize{Concentration of uninfected target cells $Z$}]{\label{statevar:Z}
\includegraphics[width=0.45\textwidth]{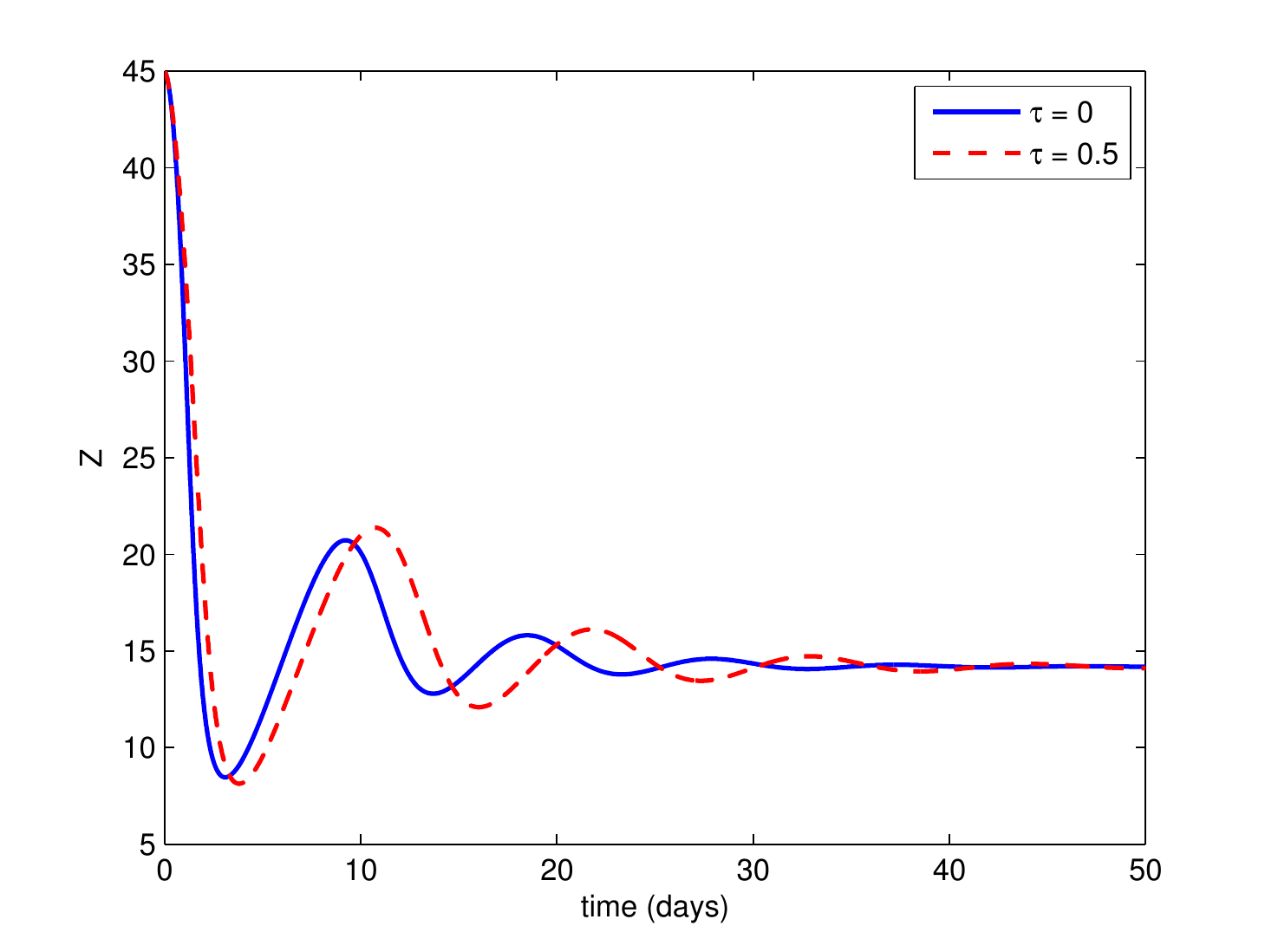}}
\subfloat[\footnotesize{Concentration of infected cells $I$}]{\label{statevar:I}
\includegraphics[width=0.45\textwidth]{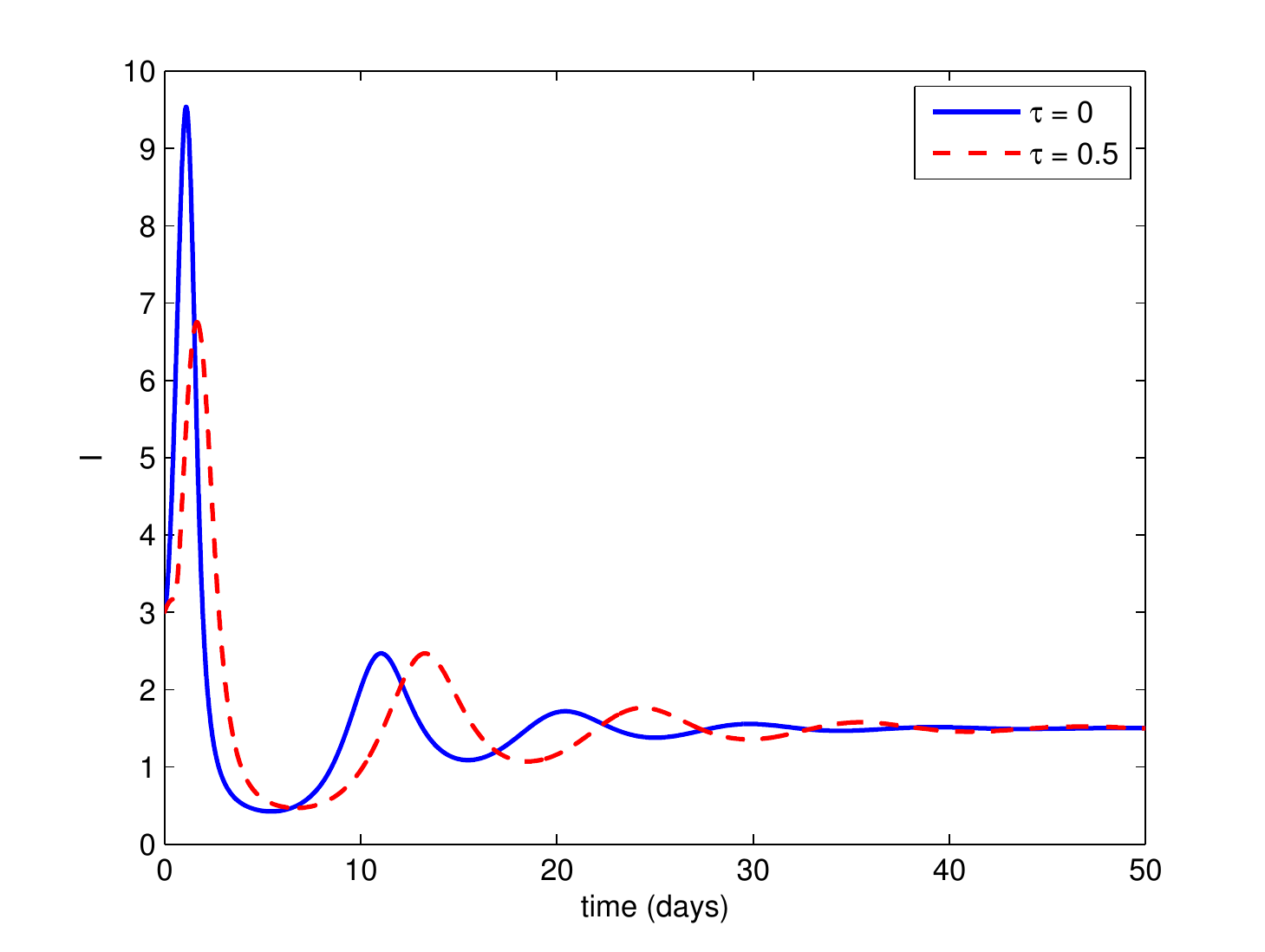}}\\
\subfloat[\footnotesize{Concentration of virus $V$}]{\label{statevar:V}
\includegraphics[width=0.45\textwidth]{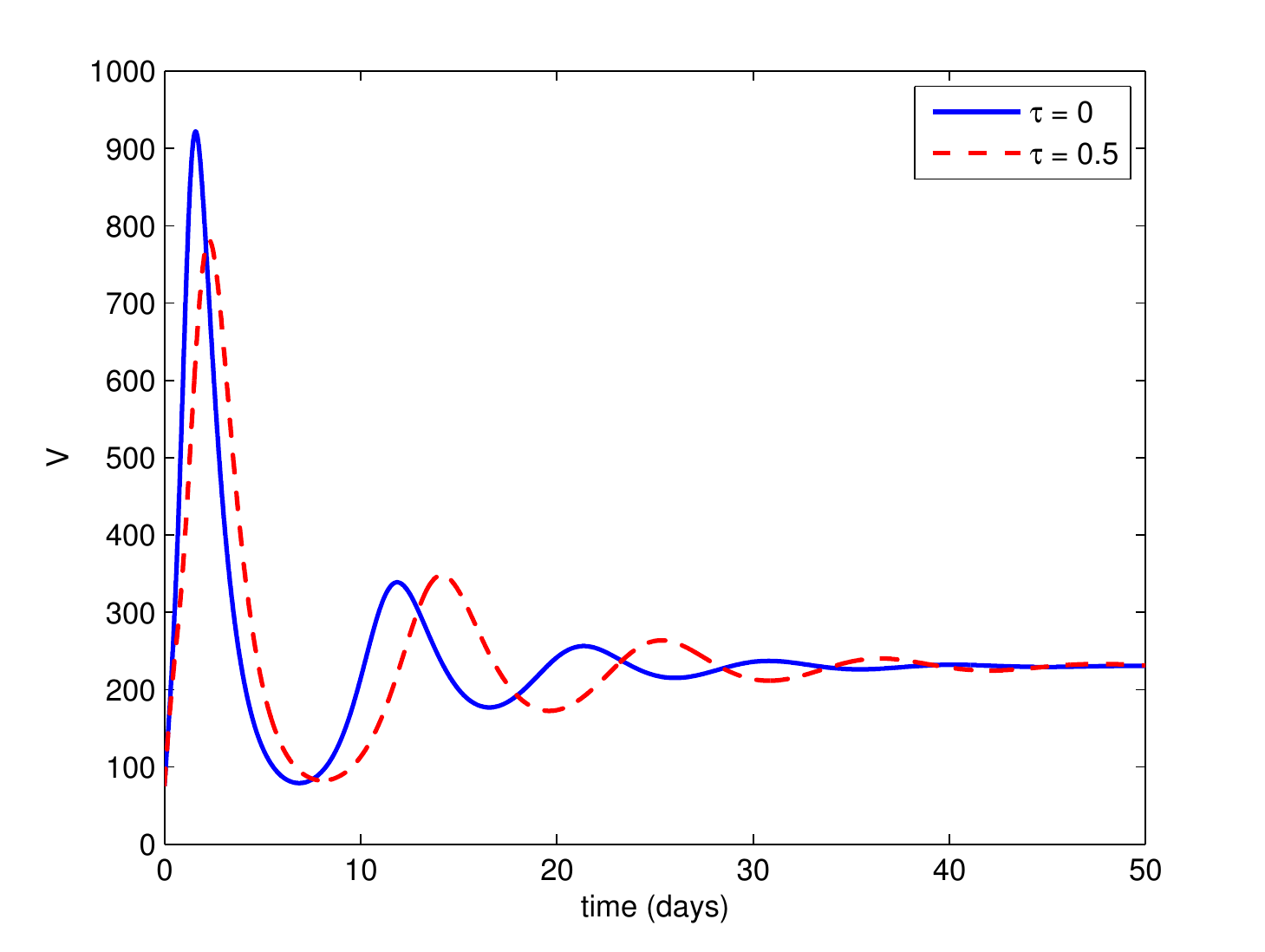}}
\subfloat[\footnotesize{Concentration of CTLs $T$}]{\label{statevar:T}
\includegraphics[width=0.45\textwidth]{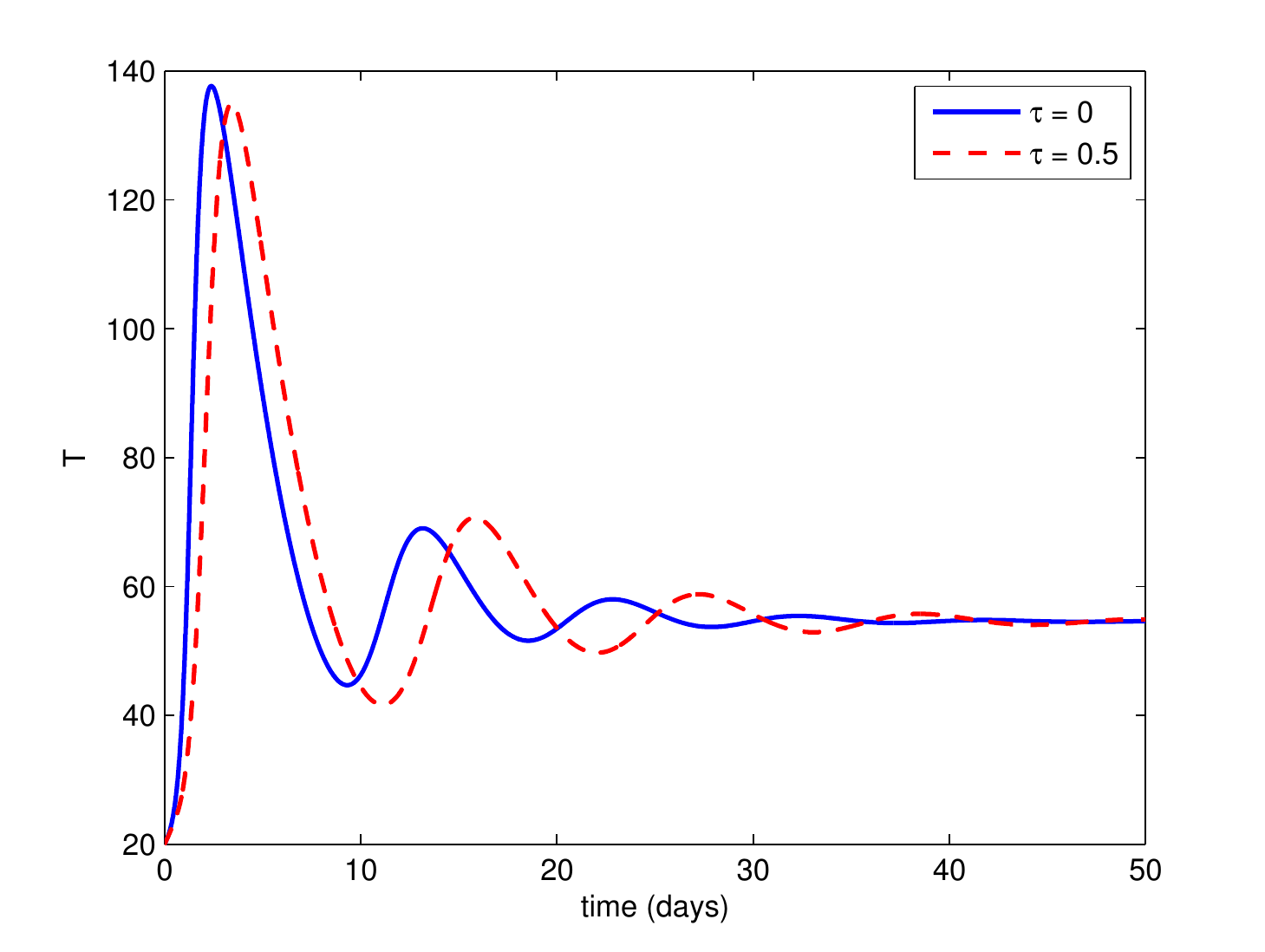}}
\caption{State variables with time delay $\tau = 0.5$ (dashed curves)
versus without delay (continuous curves).}
\label{fig:statevars:Delay:yes:no}
\end{figure}
The first local maximum of concentration of infected cells,
virus and CTLs is smaller in the delayed case ($\tau = 0.5$).
The local maxima are similar, although they
are attained at latter in the delayed case, when compared
to the nondelayed situation. At the end of 50 days,
the values of the variables $Z(t)$, $I(t)$, $V(t)$ and $T(t)$
are similar in delayed and nondelayed cases.


\subsection{Optimal control problem with state and control delays}
\label{sec:4.2}

In this section, we present numerical solutions to the delayed 
optimal control problem \eqref{eq:model:delay:control}--\eqref{costfunction}
in the time interval $[0, 50]$ days and consider three cases:
\begin{description}
\item[Case 1] $\tau = \xi = 0$ (no delays);

\item[Case 2] $\tau = 0.5$, $\xi = 0$ (intracellular delay $\tau$ only);

\item[Case 3] $\tau = 0.5$, $\xi = 0.2$ 
(intracellular delay $\tau$ and pharmacological delay $\xi$).
\end{description}
As before, we consider the parameter values from Table~\ref{table:parameter}
and the weight parameters $w=1$ and $w=5$ in the cost functional \eqref{costfunction}.
To solve the delayed optimal control problem 
\eqref{eq:model:delay:control}--\eqref{costfunction},
we discretize the control problem on a sufficiently fine grid \cite{Goellmann-Maurer-14} 
and obtain a nonlinear optimization problem (\textsc{NLP}). The \textsc{NLP} 
is implemented using the Applied Modeling Programming Language \textsc{AMPL} 
\cite{AMPL}, which can be interfaced with several large-scale nonlinear
optimization solvers like the interior-point solver \textsc{Ipopt}; see \cite{Waechter-Biegler}.
We mostly use $N=2500$ grid nodes and the trapezoidal rule as integration method 
to compute the solution with an error tolerance of $eps= 10^{-9}$.
In all three cases, the computed controls 
are bang-bang with only one switch at $t_s$:
\begin{equation}
\label{control-bang}
c(t) = \left\{
\begin{array}{rcl}
1 &&\text{for} \quad 0 \leq t <t_s ,\\[1mm]
0 &&\text{for} \quad t_s \leq t \leq 50 .
\end{array}
\right.
\end{equation}
For the weight $w=1$, we obtain the following numerical results:
$$
\begin{array}{llll}
\mbox{Case 1}: & J(c) = 475.19, & t_s = 47.08,   & Z(50) = 139.48,\\
&  I(50) =  3.6479e-02    , & V(50) = 0.96174 & T(50) = 9.70008e-06 .\\[1mm]
\mbox{Case 2}: & J(c) = 473.05, & t_s = 44.78,   & Z(50) = 139.45,\\
&  I(50) =   1.9975e-02   , & V(50) = 0.90843   , & T(50) =  9.7025e-06.\\[1mm]
\mbox{Case 3}: & J(c) = 556.70  , & t_s = 44.50 ,   & Z(50) = 139.56 , \\
&  I(50) =   1.9181e-02   , & V(50) = 0.87283   , & T(50) =  1.0822e-05.
\end{array}
$$
A zoom into the controls and switching functions, in a neighborhood 
of the switching time $t_s$, is displayed 
in Figure~\ref{Fig-control-switching-zoom}.
\begin{figure}[ht!]
\centering
\begin{tabular}{lll}
\hspace{-10mm}
\includegraphics[width=4.8cm,height=4.2cm]{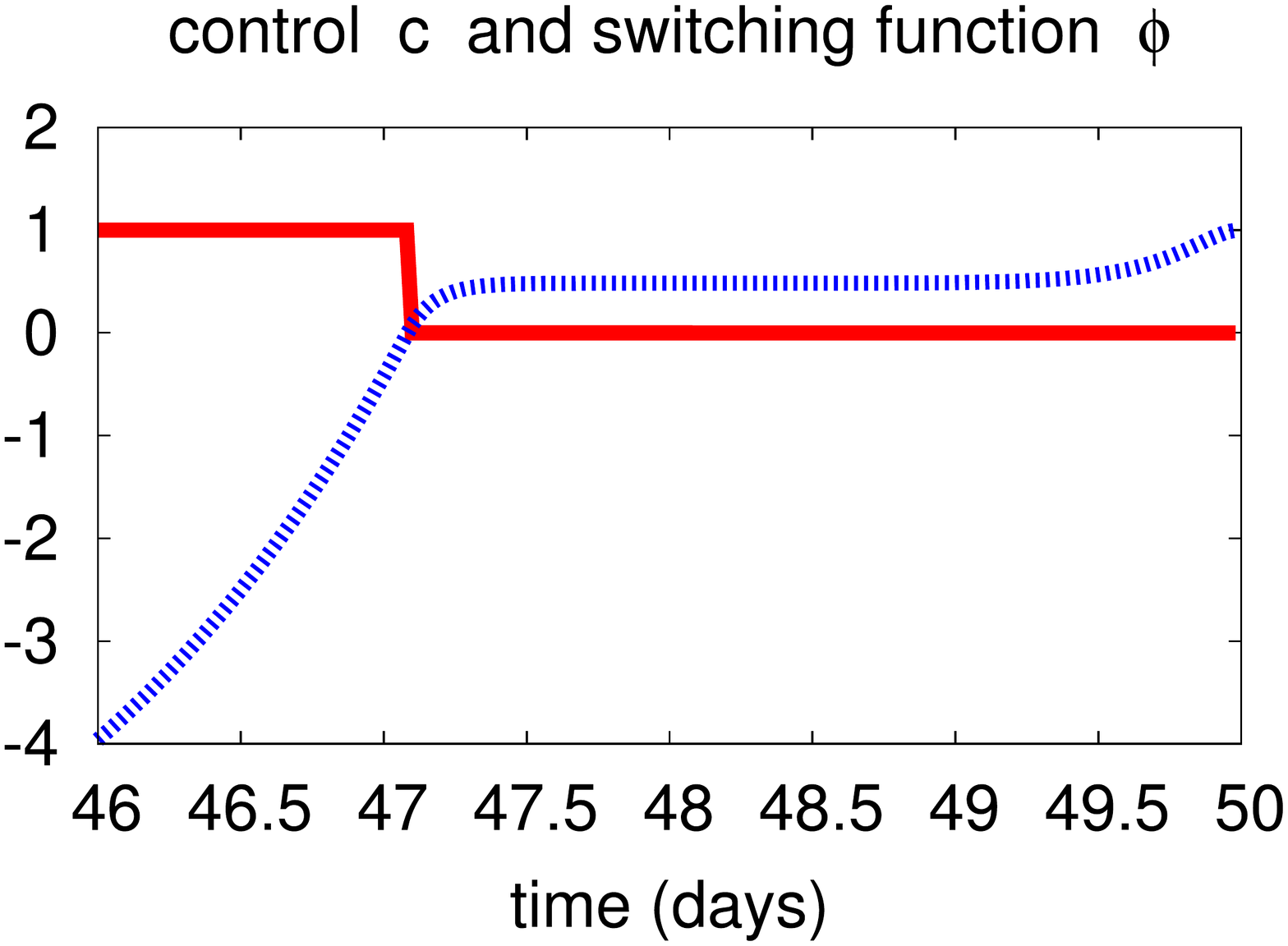}
& \hspace{-8mm}
\includegraphics[width=4.8cm,height=4.2cm]{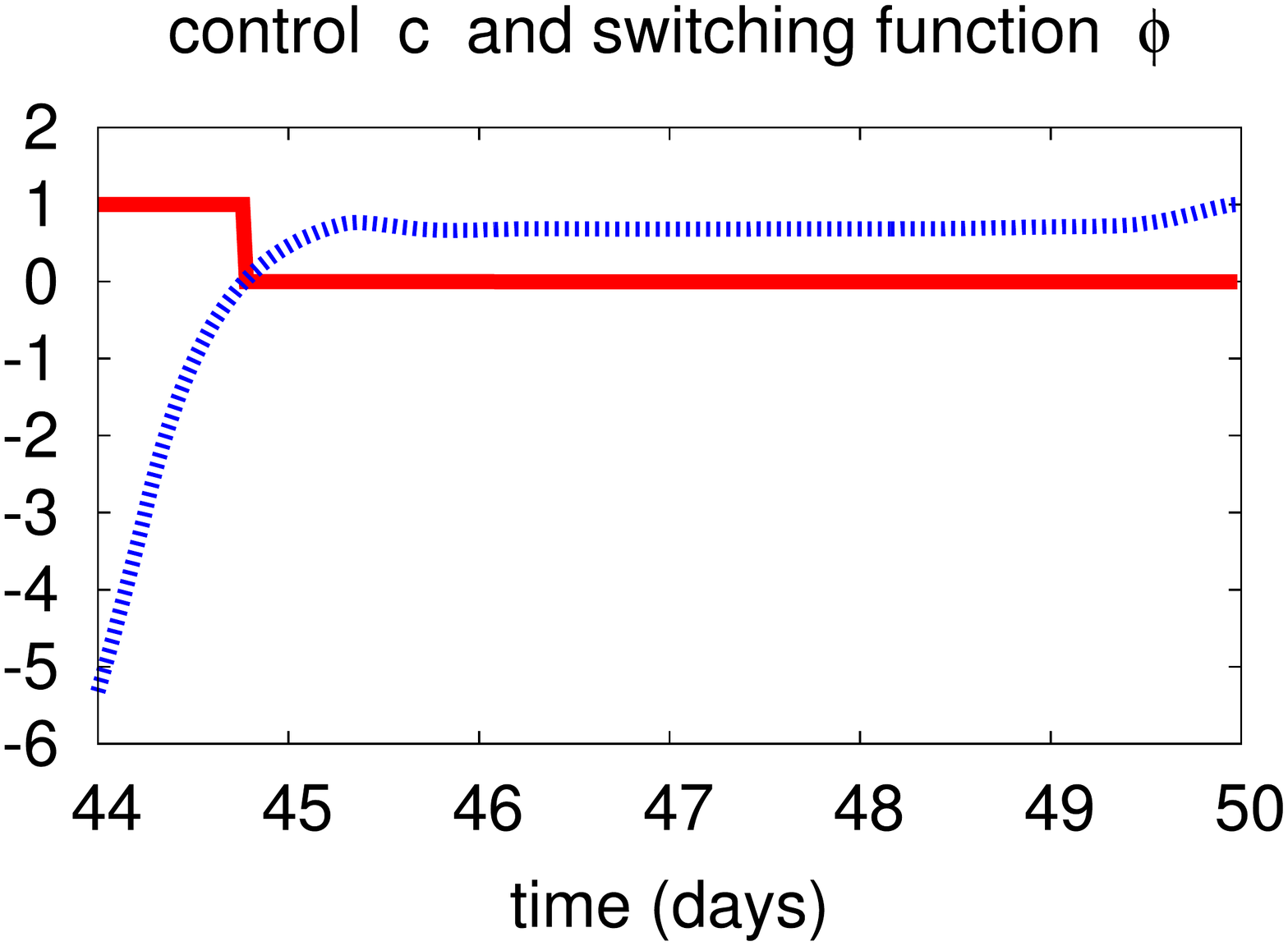}
& \hspace{-8mm}
\includegraphics[width=4.8cm,height=4.2cm]{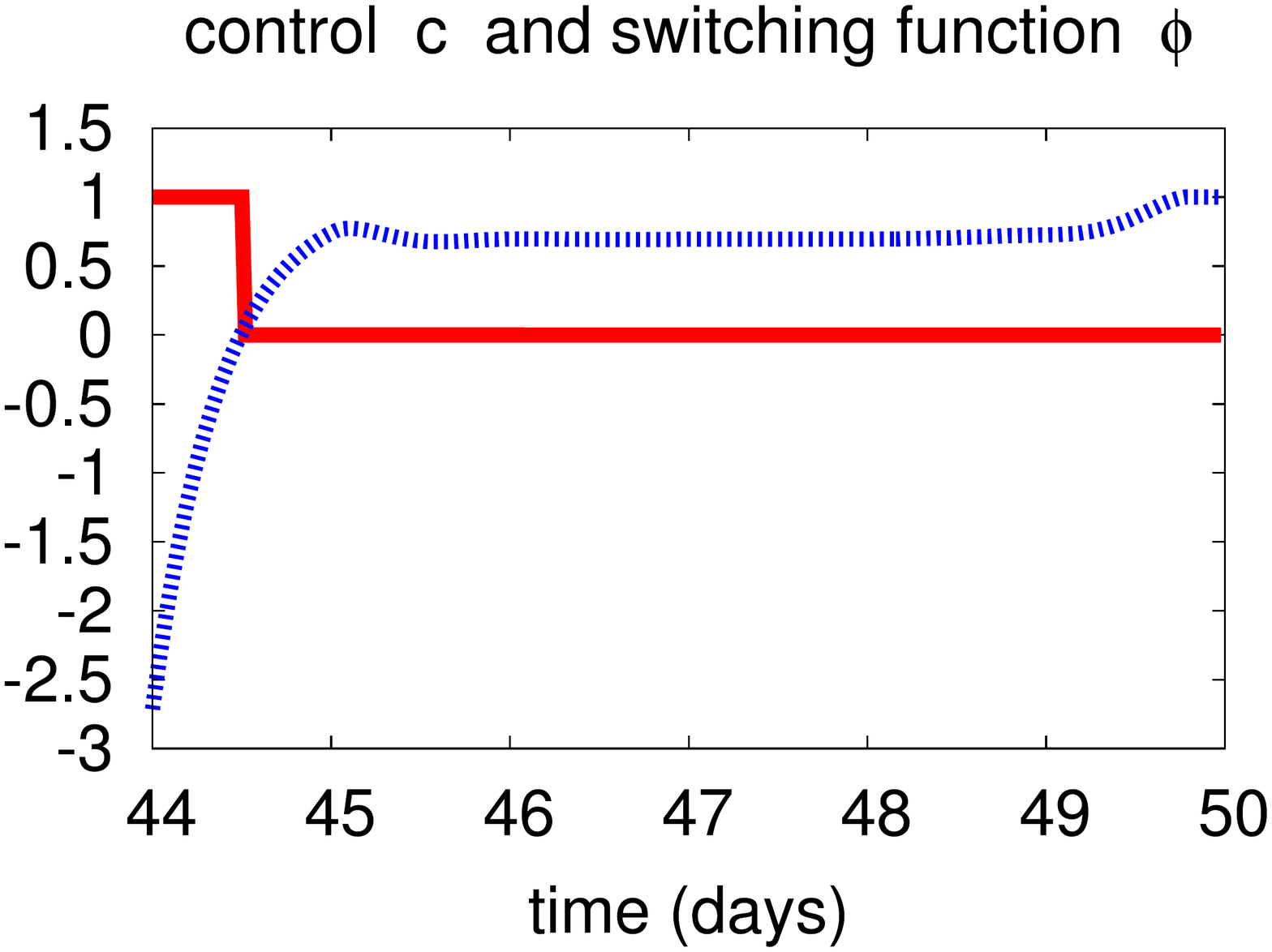}
\end{tabular}
\caption{Bang-bang control $c(t)$ \eqref{control-bang} (continuous curve)
and switching function $\phi$ \eqref{eq:sf} (dashed curve) matching the control 
law \eqref{control-law}: zoom into a neighborhood of the switching time $t_s$.\\ 
(left) Case 1, (middle) Case 2, (right) Case 3}
\label{Fig-control-switching-zoom}
\end{figure}

The state trajectories in the three cases are very similar on the terminal 
time interval $[15,50]$, while the concentration of uninfected cells 
$Z(t)$ is nearly identical on the whole time interval $[0,50]$.
To display the effect of the delays on the state variables $I,V,T$, 
Figure~\ref{state-comparison} shows a comparison of the state 
trajectories in Case 1 (no delays) and Case 3 (state and control delays).
We see that the delay in the control $c$ implies an increase
of the concentration of infected cells $I(t)$ in the first two days
(the delay on the drug effect), which is also responsible for 
an increase on the concentration of the free virus particles $V$ and CTL cells $T$.
\begin{figure}[ht!]
\centering
\begin{tabular}{lll}
\hspace{-10mm}
\includegraphics[width=4.8cm,height=4.2cm]{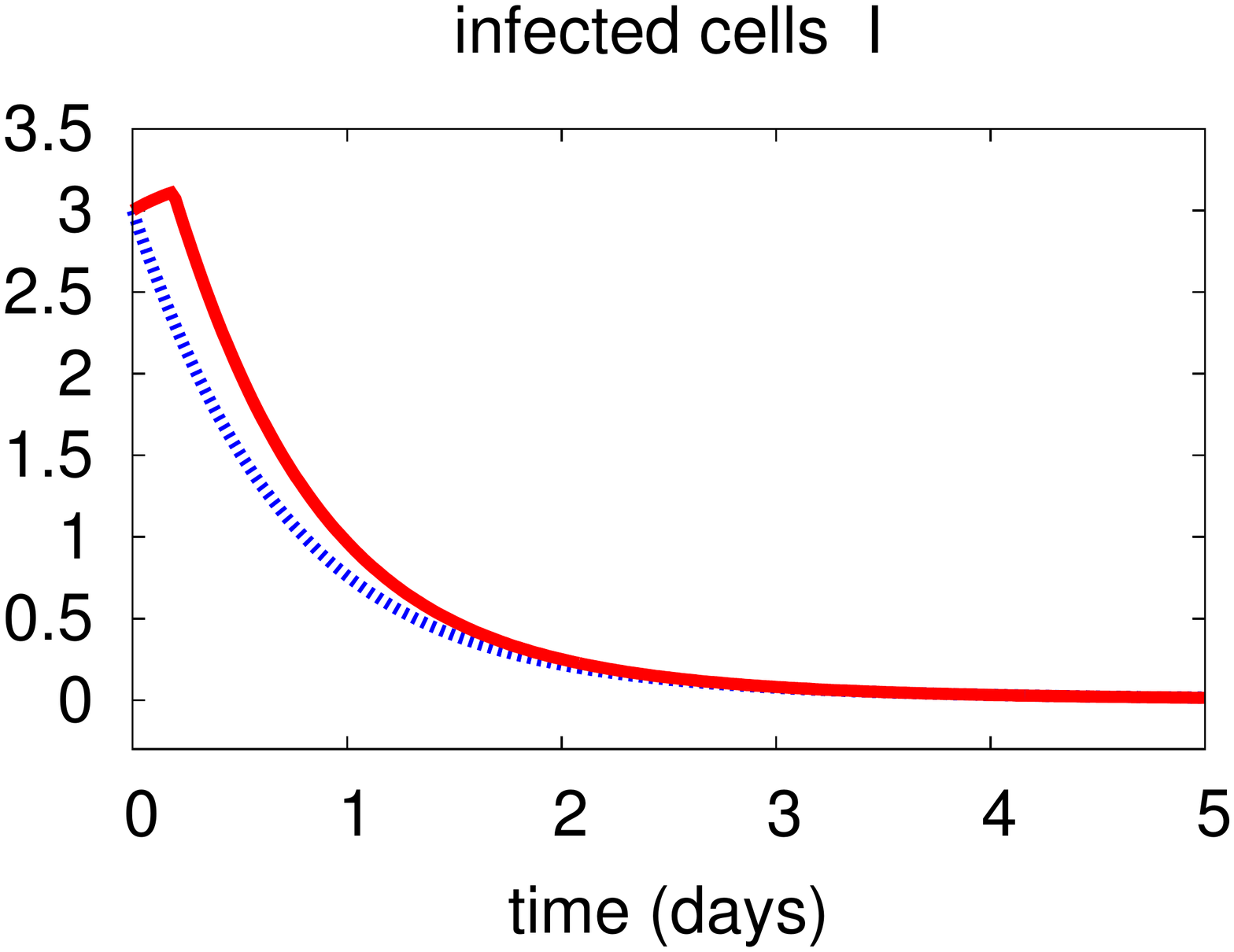}
&  \hspace{-8mm}
\includegraphics[width=4.8cm,height=4.2cm]{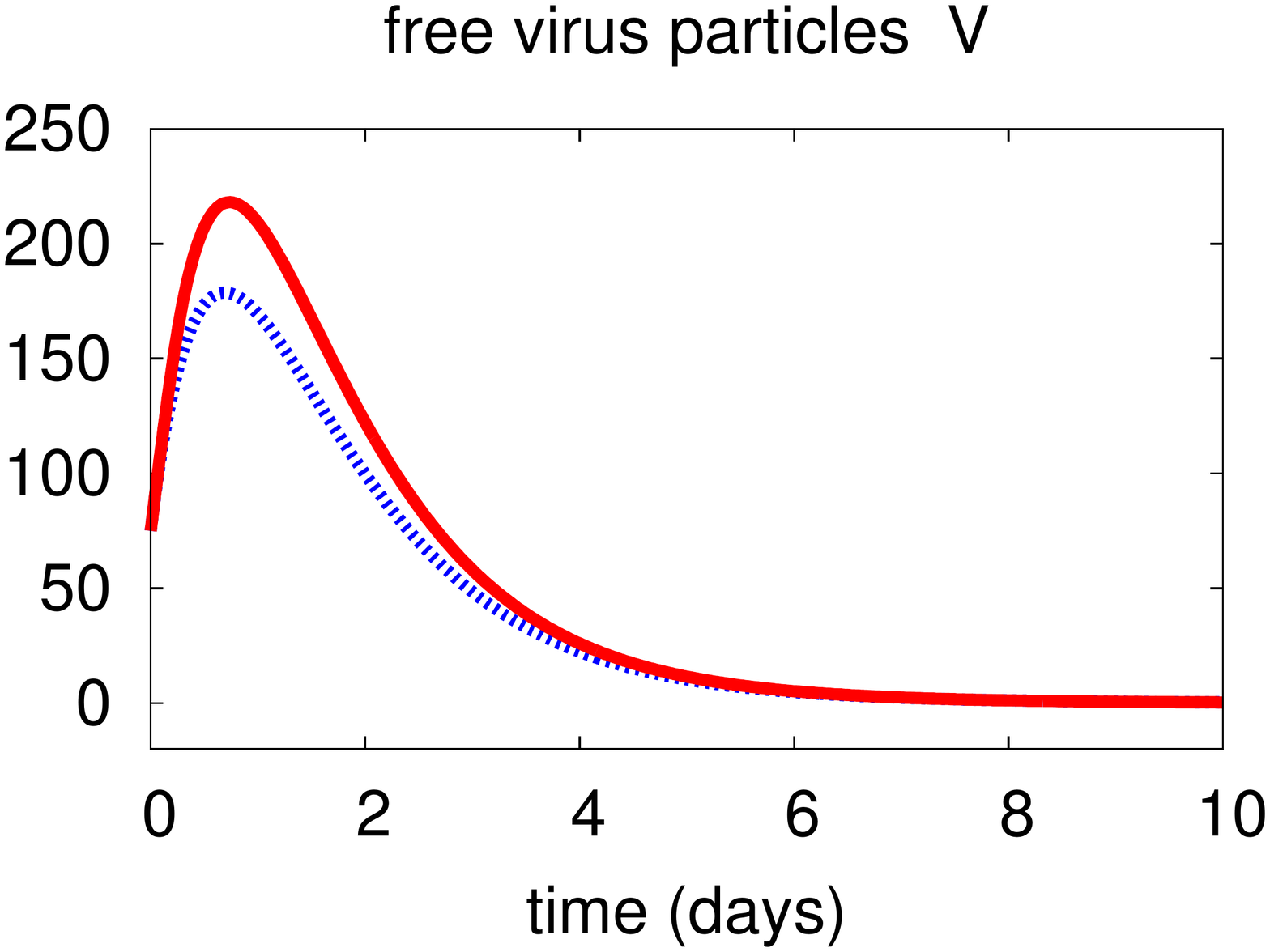}
&  \hspace{-8mm}
\includegraphics[width=4.8cm,height=4.2cm]{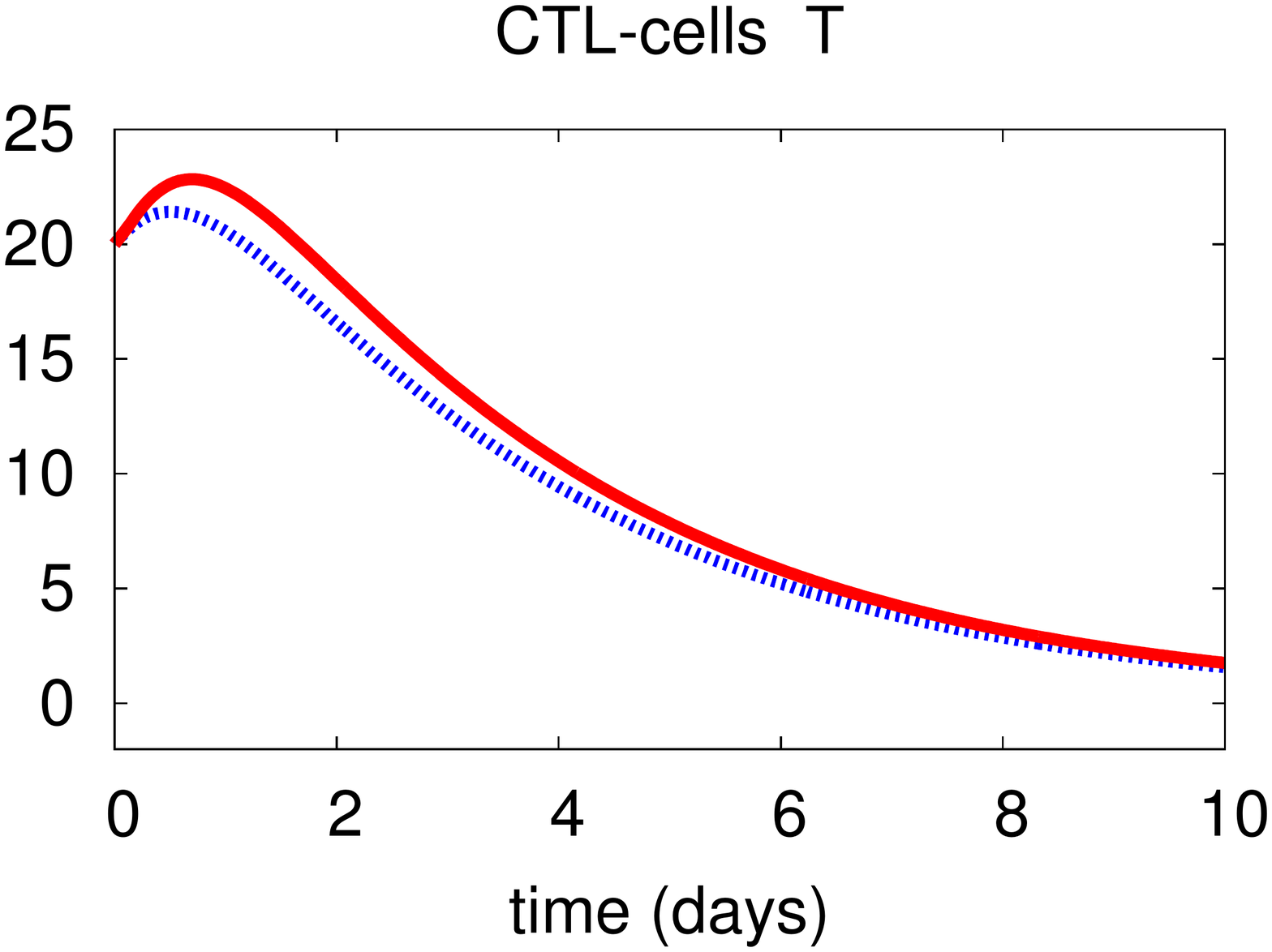}
\end{tabular}
\caption{A comparison of state trajectories in Case 1 (no delays) 
and Case 3 (delays $\tau = 0.5$ and $\xi = 0.2$). 
(left) zoom of infected cells $I(t)$ into $[0,5]$,
(middle) zoom of free virus particles $V(t)$ into $[0,10]$, 
(right) zoom of CTL cells $T(t)$ into $[0,10]$.}
\label{state-comparison}
\end{figure}

The bang-bang controls and the switching functions 
in Figure~\ref{Fig-control-switching-zoom} do not only match the
switching condition \eqref{control-law} but satisfy also the so-called 
\emph{strict bang-bang property} \cite{Osmolovskii-Maurer-Book}  
with respect to the Minimum Principle:
\begin{equation}
\label{strict-bang}
\phi(t) < 0 \quad \mbox{for} \; 0 \leq t < t_s \,, 
\quad \dot\phi(t_s) > 0 , \quad
\phi(t) > 0  \quad \mbox{for} \; t_s < t \leq 50.
\end{equation}
The strict bang-bang property enables us to check  
\emph{second-order sufficient conditions} 
(\textsc{SSC}) for the bang-bang control in the non-delayed Case 1. 
In the delayed Cases 2 and 3, no sufficient conditions are available 
in the literature. In Case 1, we consider the so-called 
Induced Optimization Problem (\textsc{IOP}), 
where the switching time $t_s$ in \eqref{control-bang} is the only optimization variable.
Hence, we optimize the function $J(t_s) = J(c) $ with respect to $t_s$.
The \textsc{IOP} can be solved using the arc-parametrization method
\cite{Maurer-etal-OCAM,Osmolovskii-Maurer-Book} and its implementation
in the optimal control package \textsc{NUDOCCCS} \cite{Bueskens}.
We obtain the highly accurate numerical results
$$
J(c) = 475.1854 , \quad  t_s =  47.090 3, 
\quad J''(t_s) = 5.0962 > 0.
$$
In view of the strict bang-bang property \eqref{strict-bang} 
and the positive second derivative $J''(t_s) = 5.0962 $,
we conclude  from Theorem~7.10 in \cite{Osmolovskii-Maurer-Book} 
that the bang-bang control in Case~1 provides a strict strong minimum. 

Since \textsc{SSC} hold, it follows from the standard sensitivity 
result in finite-dimensional optimization \cite{Fiacco} 
(cf. also \cite{Bueskens-Maurer-sens-1}) that the switching 
time $t_1$ is locally a $C^1$-function with respect to all parameters 
$p$ in the system. The state trajectories are locally $C^1$-functions 
except at the switching time $t_s$. The code \textsc{NUDOCCCS} \cite{Bueskens} 
allows to compute the \emph{sensitivity derivatives} $d t_s /dp$ and 
$dy(50)/dp$ for $y \in \{J(c),Z,I,V\}$ at a nominal parameter value $p_0$. 
The sensitivities $dT(50)/dp$ are very small so that we do not list them here.
Choosing, e.g., the parameter $p \in \{w, r, v\}$, where $w$ 
is the weight parameter in the functional \eqref{costfunction} 
and the parameters $r$ and $v$ are as in Table~\ref{table:parameter},
we get the following sensitivity derivatives at their nominal values
$w_0=1$, $r_0 = 0.0014$ and $v_0 = 1$:
$$
\begin{array}{|c|r|r|r|r|r|}
\hline
\mbox{parameter} & dt_1/dp & dJ(c)/dp & dZ(50)/dp & dI(50)/dp & dV(50)/dp \\  \hline
p = w   &  -0.1962  &  47.09  & -0.03803 & 0.03589 & 0.9464 \\ \hline
p = r   &  1146 &  1063  & -14.63 & 14.73 & 26.89 \\ \hline
p = v   &  -0.5394  &  -428.4  & -0.003583 & 0.003185 & -0.002218 \\ \hline
\end{array}
$$
The sensitivity derivatives quantify our more intuitive
feeling on how the switching time changes under parameter perturbations.
As an example, let us increase the weight parameter $w=1 + \Delta$ 
for the control in the objective $J(c)$ \eqref{costfunction}. 
Then the switching time $t_s(w)$ decreases and has the approximative value  
$$
t_s(w) \approx t_s(1) + \frac{dt_s}{dw} \Delta.
$$
Similar Taylor expansion approximations hold for the other quantities.
It is an interesting exercise to show that the sensitivity derivative 
$dJ(c)/dp$ agrees with $t_s = 47.09$.
Finally, Figure~\ref{fig:statevars:Delay:yes:Control:yes:no} displays 
a comparison of the \emph{controlled} state variables with the 
\emph{uncontrolled} ones in Case~2.
\begin{figure}[htb]
\centering
\subfloat[\footnotesize{Concentration of uninfected target cells $Z$}]{\label{statevar:Z:control}
\includegraphics[width=0.45\textwidth]{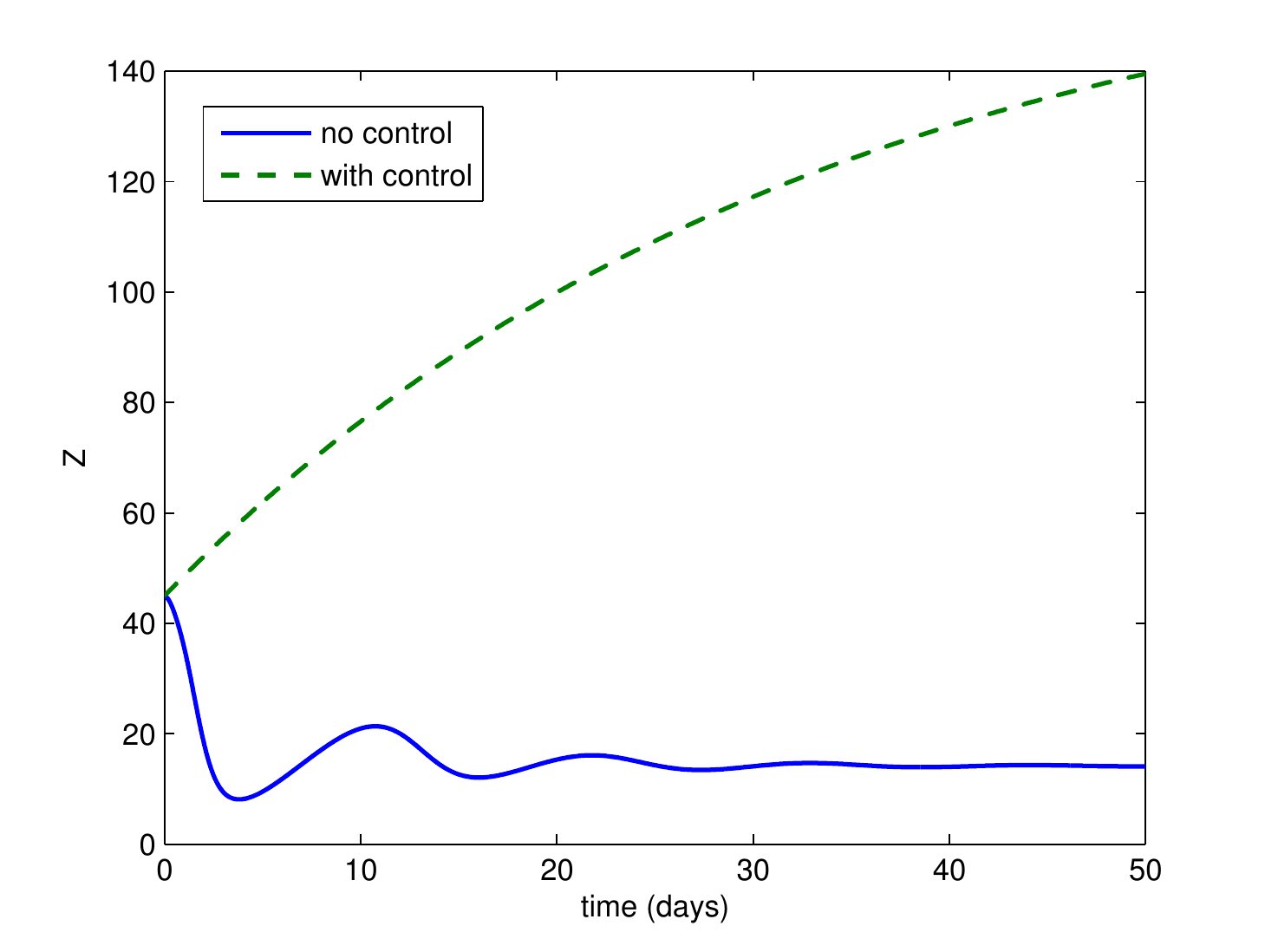}}
\subfloat[\footnotesize{Concentration of infected cells $I$}]{\label{statevar:I:control}
\includegraphics[width=0.45\textwidth]{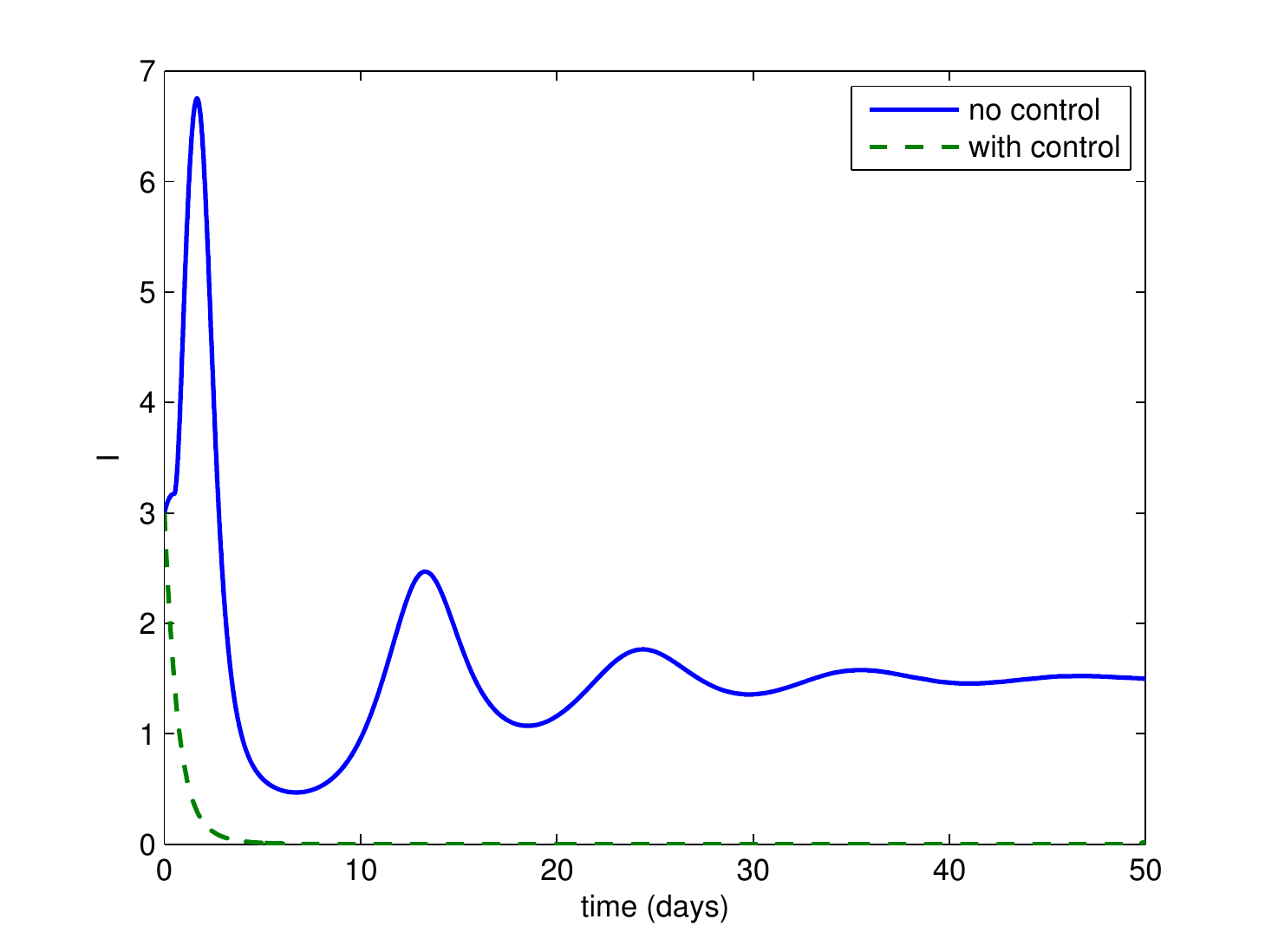}}\\
\subfloat[\footnotesize{Concentration of virus $V$}]{\label{statevar:V:control}
\includegraphics[width=0.45\textwidth]{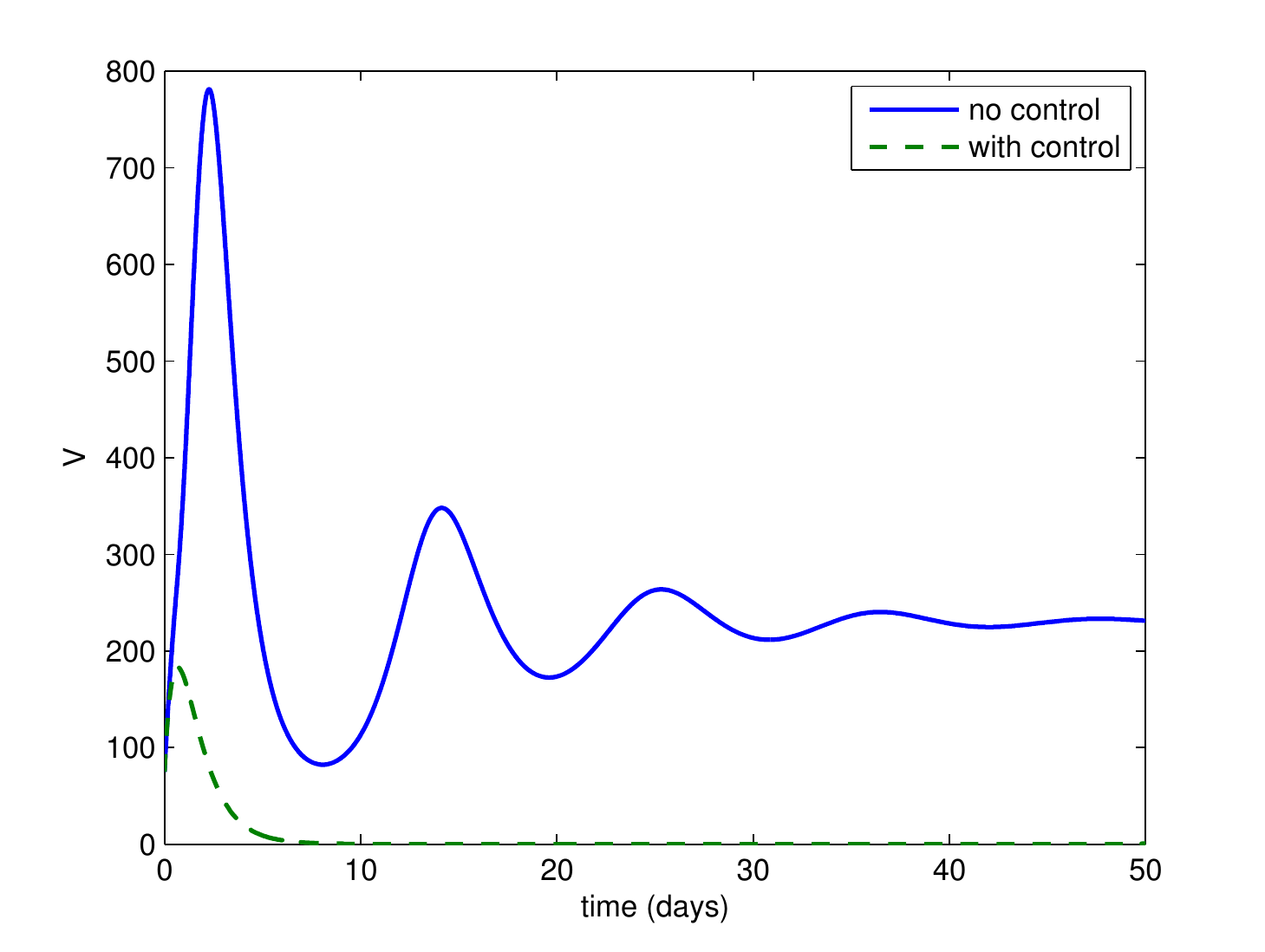}}
\subfloat[\footnotesize{Concentration of CTLs $T$}]{\label{statevar:T:control}
\includegraphics[width=0.45\textwidth]{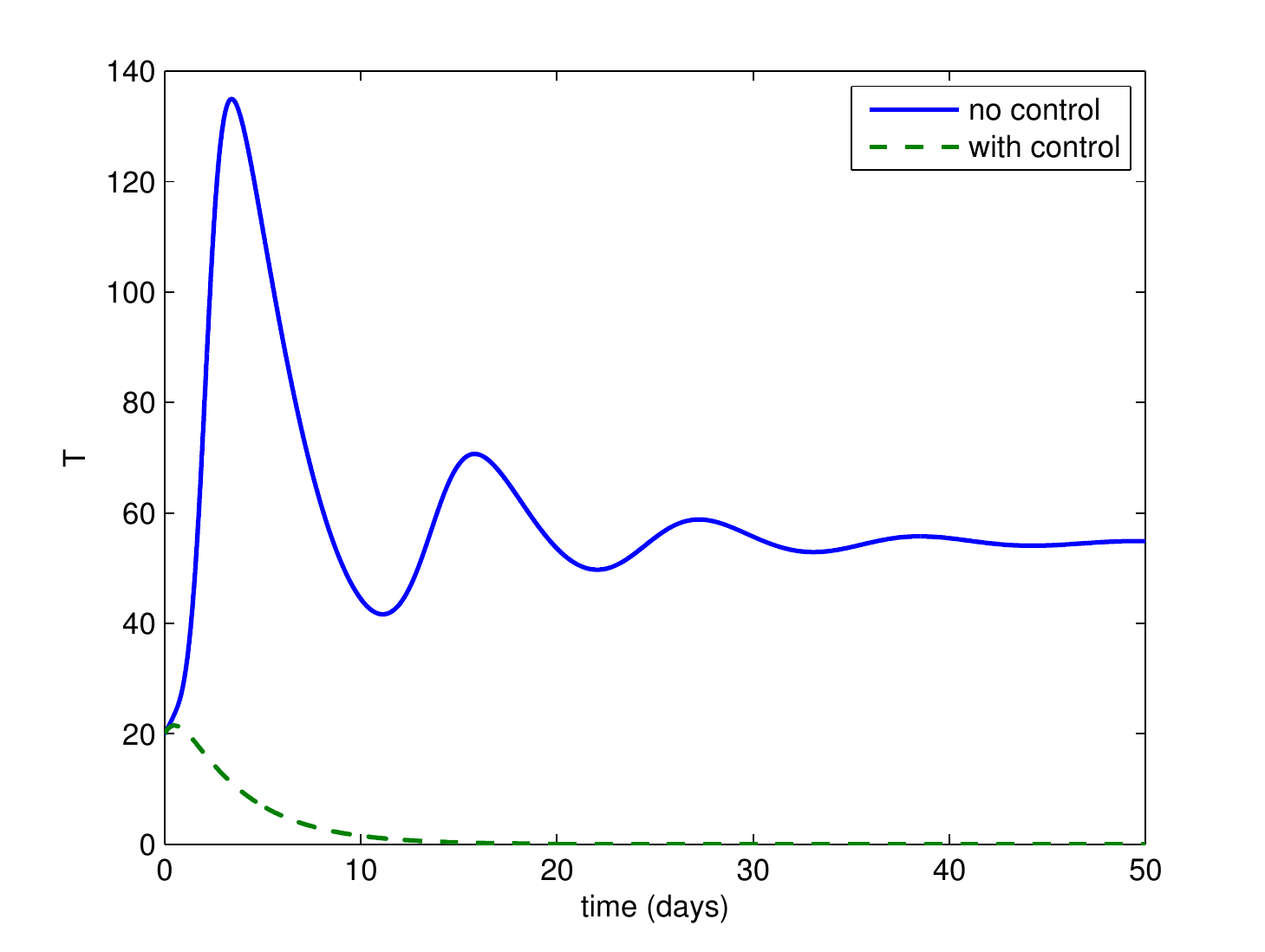}}
\caption{State variables in the case of an intracellular delay only
($\tau = 0.5$ and $\xi=0$): controlled (dashed lines)
versus uncontrolled situations (continuous lines).}
\label{fig:statevars:Delay:yes:Control:yes:no}
\end{figure}

The solutions for the control weight $w=5$ in the cost functional \eqref{costfunction}
are very similar to those for $w=1$. In all three cases, the control 
is bang-bang \eqref{control-bang} with one switching time $t_s$. 
\textsc{IPOPT} \cite{Waechter-Biegler} furnishes the following numerical results:
$$
\hspace{-10mm}
\begin{array}{lll}
\text{Case 1} : &  J(c) = 662.75, &  t_s = 46.77, \\
\text{Case 2} : & J(c) = 650.64,  &  t_s = 44.18, \\
\text{Case 3} : & J(c) = 733.19,  &  t_s = 43.88.
\end{array}
$$
The computed switching functions \eqref{eq:sf} match 
the control law \eqref{control-law} and satisfy 
the strict bang-bang property \eqref{strict-bang}.
\textsc{SSC} can only be verified in the non-delayed Case~1.
\textsc{NUDOCCCS} computes the second derivative $J''(t_s) =  25.43 > 0$.
Hence, the bang-bang control provides a strict strong minimum 
in view of \cite[Theorem~7.10]{Osmolovskii-Maurer-Book}.


\section{Conclusion and discussion}
\label{sec:discussion}

In this paper we have considered not only intracellular delay
(delay $\tau$ in the state variables) as done in the literature
\cite{Haffat:Yousfi:OCHIVdelay:2012}, but also a pharmacological delay
(delay $\xi$ in the control function). The pharmacological delay causes
an increase of the concentration of the infected cells in an initial
interval of time. However, after this increase, related to the delay
in the action of the drugs in the cells, the concentrations of infected cells,
virus and CTL cells associated to the extremal solution of the optimal control
problem with both delays in state and control variables, decrease significantly.
The extremal control is bang-bang and switches from its maximal value one to zero.
This type of control is easier to implement, from a medical point of view, 
when compared to controls found in \cite{Haffat:Yousfi:OCHIVdelay:2012}
for $L^2$ cost functionals. We observe that the extremal control derived 
from the application of Pontryagin's necessary optimality condition 
\cite[Theorem~3.1]{Goellmann-Maurer-14} to our multiple delayed 
optimal control problem, is associated to a marked reduction 
of the concentration of infected cells, virus and CTLs, 
as well as treatment costs, and to an increase
of the uninfected target cells. Sufficient optimality conditions 
could only be checked for the non-delayed solution in Case~1.
In this case, we could also perform a local sensitivity analysis by computing
the sensitivity derivatives. It remains an open and challenging question  
to prove and verify sufficient optimality conditions for delayed bang-bang controls.


\section*{Acknowledgments}

This research was supported by the
Portuguese Foundation for Science and Technology (FCT)
within projects UID/MAT/04106/2013 (CIDMA);
PTDC/EEI-AUT/2933/2014 (TOCCATA),
reference PTDC/EEI-AUT/2933/2014, funded by Project 
3599 -- Promover a Produ\c{c}\~ao Cient\'{\i}fica e Desenvolvimento
Tecnol\'ogico e a Constitui\c{c}\~ao de Redes Tem\'aticas (3599-PPCDT)
and FEDER funds through COMPETE 2020, Programa Operacional
Competitividade e Interna\-ci\-o\-na\-li\-za\-\c{c}\~ao (POCI).
Filipe Rodrigues is also supported by the FCT PhD
fellowship PD/BD/114185/2016; Silva by the FCT post-doc
grant SFRH/BPD/72061/2010. The authors are grateful
to three referees for their valuable
comments and helpful suggestions.



\medskip
Submitted July 12, 2016; revised Nov 26, 2016 and March 1, 2017;\\ accepted Aug 21, 2017.
\medskip


\end{document}